\documentclass[11pt]{amsart}
\usepackage{amsmath,amssymb, amsfonts,graphics,xcolor,amsthm,dsfont,mathabx}
\usepackage[T1]{fontenc}

\usepackage[colorlinks]{hyperref}
\hypersetup{
    colorlinks=true,
    linkcolor=blue,
    filecolor=magenta,      
    urlcolor=cyan,
    citecolor=red
}

\usepackage{marginnote}
\usepackage{csquotes}
\usepackage[pdftex]{graphicx}
\usepackage{epsfig}
\usepackage{latexsym}
\usepackage{euscript}
\usepackage{tocvsec2}
\usepackage{etoolbox, caption}
\usepackage{float}
\usepackage{marginnote}
\usepackage{latexsym}
\usepackage{epsfig, tikz}

\newtheorem{theorem}{Theorem}[section]
\newtheorem*{theorem*}{Theorem}
\newtheorem{lemma}[theorem]{Lemma}
\newtheorem{proposition}[theorem]{Proposition}
\newtheorem{corollary}[theorem]{Corollary}
\newtheorem{definition}[theorem]{Definition}
\newtheorem{notation}[theorem]{Notation}
\newtheorem{example}[theorem]{Example}
\newtheorem{remark}[theorem]{Remark}

\def\sup{{\rm sup}}

\def\tU{\rm U}

\def\HH{\mathcal H}
\def\U{\mathcal U}

\def\W{\mathcal W}
\def\bbR{\mathbb R}
\def\bbN{\mathbb N}
\def\bbC{\mathbb C}
\def\bbG{\mathbb G}
\def\bbZ{\mathbb Z}
\def\bbT{\mathbb T}
\def\bb1{\mathds 1}

\newcommand{\blue}[1]{{\color{black} #1}}

\newcommand{\green}[1]{{\color{black} #1}}

\begin{document}

\title{A noncommutative approach to the graphon Fourier transform}

\author[M. Ghandehari]{Mahya Ghandehari}
\address{Department of Mathematical Sciences,
Department of Mathematical Sciences, 
University of Delaware, Newark, DE, USA, 19716}
\email{mahya@udel.edu}

\author[J. Janssen]{Jeannette Janssen}
\address{Department of Mathematics \& Statistics, 
Dalhousie University, Halifax, Nova Scotia, Canada, B3H 3J5}
\email{jeannette.janssen@dal.ca}

\author[N. Kalyaniwalla]{Nauzer Kalyaniwalla}
\address{Faculty of Computer Science, 
Dalhousie University, Halifax, Nova Scotia, Canada, B3H 3J5}
\email{nauzerk@cs.dal.ca}

\date{13 April 2022}

\begin{abstract}
Signal analysis on graphs relies heavily on the graph Fourier transform, which is defined as the projection of a signal onto an eigenbasis of the associated shift operator. Large graphs of similar structure may be represented by a graphon. Theoretically, graphons are limit objects of converging sequences of graphs. 
Our work extends previous research proposing
a common scheme for signal analysis of graphs that are similar in structure to a graphon. We extend a previous definition of graphon Fourier transform, and show that the graph Fourier transforms of graphs in a converging graph sequence converge to the graphon Fourier transform of the limiting graphon.
We then apply this convergence result to signal processing on Cayley graphons. 
We show that Fourier analysis of the underlying group enables the construction of a suitable eigen-decomposition for the graphon, which can be used as a common framework for signal processing on  graphs converging to the graphon.\\


\end{abstract}

\maketitle

\tableofcontents

\section{Introduction}\label{sec:intro}

The proliferation of networked data has motivated the extension of traditional methods of signal processing to signals defined on graphs. An important operational tool of traditional signal processing is the Fourier transform, which allows for the analysis of signals via their spectral decomposition. The extension of this approach to graphs has been used {successfully} to analyze data on networks,  such as sensor networks, or networks of interactions of chemicals in a cell or of financial transactions.

A graph signal is any complex-valued function on the vertices of the graph, which may represent pollution levels measured by a sensor network, or neural activity levels in regions of the brain  (see \cite{2018:Ortega:GSPOverview} and the references therein). Important properties of signals, such as diffusion of a signal, noise reduction or the concept of smoothness, can be modeled by means of a matrix representing the graph structure.
Such a matrix is referred to as the graph \emph{shift operator}, and is typically taken to be the graph adjacency matrix or the graph Laplacian. A graph Fourier transform can then be defined as the projection of the signal onto an eigenbasis of the graph shift operator. 
See \cite{2018:Ortega:GSPOverview,SandryhailaMoura13} for an overview of recent developments in graph signal processing.

Taking this approach, we can see that the processing of a graph signal rigidly depends on the underlying network. 
Any change in the network requires a re-computation of the eigenbasis of the graph shift operator, which is  expensive. Further, in many cases networks evolve over time, resulting in a sequence of graphs that are structurally similar. These issues heighten the need for an instance-independent framework for handling processing of signals over large graphs. 
In \cite{lovaszszegedy2006}, \emph{graphons} were introduced as limit objects of converging sequences of dense graphs 
\blue{(i.e., graphs $G_n$ that have $\Omega(|V(G_n)|^2)$ edges).}
Graphons retain the large-scale structure of the families of graphs they represent, and are hence useful representatives of these families. 
The use of graphons to guide graph signal processing was first proposed in \cite{MorencyLeus17}, and a graphon Fourier transform was proposed in \cite{ruiz2}.

A graphon is commonly defined as a symmetric, measurable function $w: [0,1]^2 \rightarrow [0,1]$.
More generally, however, one can define a graphon on any standard probability space (\cite[Corollary 3.3]{Borg-Chayes-Lovasz-2010}).\footnote{\blue{A standard probability space is obtained from a  Borel probability space by enlarging its sigma-algebra to include all subsets of its nullsets. A Borel probability space is a probability space that is isomorphic (up to nullsets) to the disjoint union of a closed interval (with the Borel sigma-algebra and the Lebesgue measure) and a countable set of atoms (\cite[Section A.3.1]{lovasz-book}).}}
For practical reasons, and especially to deal  with Cayley graphons in Section \ref{sec:cayley}, we define graphons and graphon signals on general standard probability spaces as follows.
\begin{definition}\label{def:graphon-signal}
Let  $(X,\mu)$ be a standard probability space, and $L^2(X)$ be the associated space of square-integrable functions. 
A function $w:X\times X\to [0,1]$ is called a graphon represented on $X$ if $w$ is measurable and symmetric (i.e.~$w(x,y)=w(y,x)$ almost everywhere). 
A graphon signal on $w$ is a pair $(w,f)$, where $f:X\rightarrow \bbC$ belongs to $L^2(X)$. 
\end{definition}
Morency and Leus in \cite{MorencyLeus17} investigate asymptotic behavior of signal processing for sequences of graph signals converging to a graphon signal.
They use the limit theory of dense graph sequences (\cite{lovaszszegedy2006}) to define the notion of converging graph signals. 
Their original definition is for graphons on $[0,1]$, but naturally extends to the general form of graphons as follows:
Let $w$ be a graphon on a standard probability space $(X,\mu)$.
A sequence of graph signals $\{(G_n, f_n)\}$ is said to converge to the graphon signal $(w , f)$ if there exists a labeling for each  $G_n$ such that
$\|G_n-w\|_\Box\rightarrow 0$ and $\|f^X_n-f\|_2\rightarrow 0$,
where $f^X_n$ is the natural representation of the graph signal $f_n$ as a simple function in $L^2(X)$;
see Section \ref{sec:generalize} for the details of this definition, and Section \ref{subsec:graphon} for the definition of cut norm $\|\cdot\|_\Box$.

Ruiz, Chamon and Ribeiro in \cite{ruiz2}, and more extensively in \cite{RuizChamonRibeiro21}, give a convergence result for the graphon Fourier transform. The main result of \cite{RuizChamonRibeiro21} is that for a sequence of graph signals $\{(G_n,f_n)\}$ that converge to a graphon signal $(w,f)$, the graph Fourier transform of $(G_n,f_n)$ converges to the graphon Fourier transform of $(w,f)$ if the following conditions on $w$ and $f$ are satisfied: 
\begin{itemize}
\item[(i)] the graphon signal $f$ is $c$-bandlimited, 
\item[(ii)] $w$ is a non-derogatory graphon (i.e.~the integral operator associated with $w$ does not have repeated nonzero eigenvalues.)
\end{itemize}
\blue{We generalize this result in Section \ref{sec:generalize}.} Namely, we drop both conditions which were imposed on $w$ and $f$, allowing for a convergence theorem \blue{applicable to all graphons} (see Theorem~\ref{thm:convergence} for a precise statement of our result). 
Even taking the density of non-derogatory graphons in the space of all graphons into account (\cite[Proposition 1]{RuizChamonRibeiro21}), our result genuinely extends the above-mentioned theorem of Ruiz \emph{et al}. Indeed, a linear operation (e.g.~graph Fourier transform in this context) may be continuous on a dense subset (e.g.~the collection of non-derogatory graphons), but fail to be continuous on the whole space (e.g.~all graphons). We note that many important examples of graphons, including many Cayley graphons on non-Abelian groups, have multiple non-zero eigenvalues and thus do not meet the criteria for convergence as obtained by Ruiz \emph{et al}.

To achieve the continuity of the graph Fourier transform on the entire space of graphons, we need to {refine the definition of graph/graphon Fourier transform}.
The theory of signal processing on graphs was originally inspired by Fourier analysis on ${\mathbb Z}_N$, or more generally, harmonic analysis on Abelian groups. Any shortcomings of this approach may be attributed to the fact that graphs lack the structural symmetries of ${\mathbb Z}_N$.
Inspired by Fourier analysis of non-Abelian groups, we replace the concept of ``Fourier coefficients'' by projections onto eigenspaces of the shift operator. This point of view enables us to deal with eigenvalues with multiplicity higher than 1. We then express continuity of the graph Fourier transform in terms of convergence of such projections in suitable norms. We note that every nonzero eigenvalue of a graphon $w$ has finite multiplicity, and the associated projection onto its eigenspace is a finite-rank operator. So our result in Theorem~\ref{thm:convergence} is of the form of convergence of matrices (of finite but increasing size) to finite-rank operators. 

\blue{Our original motivation for establishing Theorem~\ref{thm:convergence} was to develop signal processing on \emph{Cayley graphons},  taking the Fourier analysis of the underlying group into account. We devote Section~\ref{sec:cayley} to this task.}
Cayley graphons, first introduced in \cite{cayley-graphon}, are defined on a (finite or compact) group $\bbG$ according to a connection function defined over $\bbG$. 
Cayley graphons form a natural extension of Cayley graphs.
Generally, graphs sampled from a Cayley graphon are not themselves Cayley graphs. Instead, they can be seen as ``fuzzy versions'' of Cayley graphs, which preserve the symmetries of the group on a large scale, but are locally random. 
In Section~\ref{sec:cayley}, we will show that the group structure of a Cayley graphon may be used to develop a specific framework for signal processing on the graphon, which can then be used to provide an instance-independent framework for graph signal processing.

\blue{The group symmetries make Cayley graphs/graphons appropriate models for many networks.
A prime example of this phenomenon can be seen in ranked data analysis, which has applications in various areas such as image processing (\cite{model-fitting}), recommender systems (\cite{WangSE14}), bioinformatics (\cite{LiWX19,UminskyBGGN19}), and computer vision (\cite{HuangGG09}).
Ranked data can be naturally modeled as signals on Cayley graphs of the permutation group ${\mathbb S}_n$.
The vertex set of such a graph represents all preference rankings of $n$ objects or
candidates. The choice of the generating set for the Cayley
graph formalizes the idea of distance between rankings. An important example of such Cayley graphs is the
permutahedron, where the generating set is the set of all adjacent transpositions. In \cite{permutahedron}, Chen et al.~develop signal processing on the permutahedron, and propose a novel frame construction for analyzing ranked data. 
Moreover, they show that the frame atoms have an interpretable structure, 
in the sense that the analysis coefficients with respect to this frame provide
meaningful information for ranked data. Examples of
interpretations of the analysis coefficients include popularity of
candidates, whether a candidate is polarizing, or whether two
candidates are likely to be ranked similarly. 
For more examples of signal processing on Cayley graphs, see \cite{Rockmore1} for fast Fourier transforms on certain Cayley graphs, and \cite{GGH-frames} for construction of tight windowed Fourier frames for Cayley graphs.}
 
\blue{Cayley graphons can be viewed as the blueprints of graphs which are ``nearly'' Cayley. 
In particular, a Cayley graph can be turned into a Cayley graphon that is piece-wise constant.  Samples from such a graphon will have sets of vertices with similar linking behaviour. In Example \ref{exp:S3} we study a graphon based on a small permutahedron. Vertices in the sampled graphs are partitioned so that each part represents a different permutation.  These graphs can be interpreted as fuzzy versions of the permutahedron.
When the group underlying a Cayley graphon is the torus, we obtain a model that is based on a circular layout of the vertices. 
In Example \ref{exp:Watts-Strogatz} we will see such a graphon. The stochastic graph model represented by the graphon is the well-known Watts-Strogatz ``small world'' model to generate graphs with small diameter and high local clustering. } 

\blue{Providing a framework for signal processing of Cayley graphons that is transferable to the sampled graphs can be very useful in practical applications.}
As a critical step towards developing a framework for signal processing on Cayley graphons, we study the spectral decomposition of Cayley graphons in Section~\ref{sec:cayley}. Using the representation theory of the underlying group,  we first show how to derive eigenvalues and eigenvectors of Cayley graphons. {Namely, in  Theorem~\ref{thm:eigenvector}, we show that the eigenvalues/eigenvectors of a Cayley graphon can be derived from the eigenvalues/eigenvectors of the irreducible representations of the underlying group $\bbG$ applied to the function on $\bbG$ that defines the graphon.} 
Proposition~\ref{prop:basis} then provides us with a basis for $L^2({\bbG})$ which diagonalizes the Cayley graphon (or more precisely, its associated integral operator). The proposed basis is obtained as a combination of coefficient functions of irreducible representations of ${\bbG}$. These are important functions associated with irreducible representations of a locally compact group, and play a central role in the harmonic analysis of non-Abelian groups.
Proposition~\ref{prop:basis} can be applied to Cayley graphs; in that case, it offers an improvement over many earlier results on calculating the eigenvalues and eigenvectors of Cayley graphs, as such results often work under the extra assumption that the generating set of the Cayley graph is closed under conjugation. 
The bases as defined in Proposition~\ref{prop:basis} can be used as a framework for signal analysis of graphs whose structure conforms with the Cayley graphon.

\section{Notations and background}
\subsection{Signal processing on graphs}\label{subsec:graphs}
Let $G$ be a graph on the vertex set $V=\{v_1,\ldots,v_N\}$. A graph signal on $G$ is a function $f:V\to \bbC$, which can also be identified  with a column vector $\left(f(v_1), f(v_2), \cdots, f(v_N)\right)^\top\in \bbC^N$, where $\top$ denotes the transpose of a vector. 
Given a graph $G$ on $N$ nodes, a graph Fourier transform can be defined as the representation of signals on an orthonormal basis for $\bbC^N$ consisting of eigenvectors of the graph shift operator (i.e.~the adjacency matrix or the graph Laplacian). In the present work, we focus our attention on signal processing using the adjacency matrix as the shift operator. 

For the rest of this article, let $A$ denote the adjacency  matrix of a given graph $G$ on $N$ vertices, and fix an orthonormal basis of eigenvectors $\{\phi_i\}_{i=1}^{N}$ associated with eigenvalues $\lambda_1\leq \lambda_2\leq \ldots\leq \lambda_{N}$ of $A$.  The \emph{graph Fourier transform} of a graph signal $f:V\to \bbC$ is defined as the expansion of $f$ in terms of the orthonormal basis $\{\phi_i\}_{i=1}^{N}$. That is,
\begin{equation}\label{GFT}
\widehat{f}(\phi_i) = \langle f, \phi_i\rangle_{\bbC^N} = \sum_{n=1}^N f(v_n)\overline{\phi_i(v_n)},
\end{equation}%
and in this setting the \emph{inverse graph Fourier transform} is given by
\begin{equation}\label{GIFT}
f(v_n) = \sum_{i=1}^{N} \widehat f (\phi_i) \phi_i(v_n).
\end{equation}
The notation  $\widehat{f}(\lambda_i)$ is also used for  $\widehat{f}(\phi_i)$. To avoid possible confusions when $A$ has repeated eigenvalues, we mainly use the notation in \eqref{GFT}.
The above definition of the graph Fourier transform generalizes the classical \emph{discrete Fourier transform}.
Please refer to ~\cite{2013:Sandryhaila:DS,2014:Sandryhaila:BD,SNFOV} for a detailed background on the graph Fourier transform, and ~\cite{2018:Ortega:GSPOverview} for a general overview of graph signal processing.

\subsection{Fourier analysis on compact (not necessarily Abelian) groups}
In Section \ref{sec:cayley}, we develop a framework for signal processing on Cayley graphons, which relies on the Fourier analysis of the underlying group. 
In this section, we give the necessary background on Fourier analysis {of compact groups},
\blue{i.e., topological groups for which the underlying topology is compact and Hausdorff.}

Let $\bbG$ be a compact (not necessarily Abelian) group equipped with its Haar measure. Let $L^2(\bbG)$ denote the set of measurable complex-valued functions on $\bbG$, identified up to sets of measure 0, that are square-integrable. The Banach space $L^2(\bbG)$ forms a Hilbert space when equipped with the inner product $\langle f,g\rangle=\int_G f\overline{g}$. Let $L^1(\bbG)$  denote the set of measurable complex-valued functions on $\bbG$, identified up to sets of measure 0, that are integrable. For $f,g\in L^1(\bbG)$, we define their convolution product to be
\begin{equation}\label{def:convolution}(f*g)(x)=\int_{\bbG} f(y)g(y^{-1}x)\, dy.
\end{equation}
This convolution product turns $L^1(\bbG)$ into a Banach algebra. The Hilbert space $L^2(\bbG)$ and the Banach algebra $L^1(\bbG)$ play a central role in harmonic analysis of (locally) compact groups.

Given a Hilbert space $\HH$, a map $\pi:\bbG\to \U(\HH)$ is called a \emph{unitary representation} of $\bbG$ if $\pi$ is a group homomorphism 
into the group of unitary operators on $\HH$, denoted by $\U(\HH)$. If $\HH={\mathbb C}^n$, we say $\pi$ is an $n$-dimensional representation. Two representations $\pi:\bbG\to \U(\HH_\pi)$ and $\rho:\bbG\to\U(\HH_\rho)$  are said to be \emph{unitarily equivalent} if there exists a unitary map $U:\HH_\pi\to\HH_\rho$ such that 
\[ U \pi(g) U^* = \rho(g), \quad \forall g\in \bbG.\]%
 A representation $\pi$ is called \emph{irreducible} if it does not admit any non-trivial closed $\pi$-invariant subspaces. 
The collection of (equivalence classes of) all irreducible representations of $\bbG$  is denoted by $\widehat{\bbG}$. In the special case where $\bbG$ is Abelian, every irreducible representation is 1-dimensional and the set $\widehat{\bbG}$, together with point-wise multiplication, forms a group called the \emph{dual group} of $\bbG$. For compact groups, every irreducible representation is finite dimensional. However, a compact group $\bbG$ may have representations of different dimensions. Therefore, in the case of a non-Abelian compact group, the collection $\widehat{\bbG}$ does not form a group.

\subsubsection{Fourier and inverse Fourier transforms for compact groups.} \label{subsec:FT-nonAbelian}
Let $f\in{L^1(\bbG)}$. The Fourier transform of $f$, denoted by ${\mathcal F}f$, is defined as a matrix-valued function on $\widehat{\bbG}$. At each $\pi\in\widehat{\bbG}$, the Fourier transform ${\mathcal F}f(\pi)$ is 
a matrix of dimensions $d_\pi\times d_\pi$  with entries in $\bbC$, defined as  
\begin{equation}\label{F-noncommutative}
{\mathcal F}f(\pi):=\pi(f)=\int_{\bbG} f(y)\pi(y)\, dy. 
\end{equation}
Here the integral is taken with respect to the Haar measure of $\bbG$, and is interpreted entry-wise (i.e.~in the weak sense). 
The inverse Fourier transform is given by
\begin{equation}\label{inverse-F-noncommutative}
f(x)=\sum_{\pi\in {\widehat{\bbG}}}d_\pi{\rm Tr}(\pi(x)^*\pi(f)),
\end{equation}
where the equality holds in $L^2(\bbG)$.
We should note that the Fourier and inverse Fourier transforms given in \eqref{F-noncommutative} and \eqref{inverse-F-noncommutative} are defined slightly differently in \cite{1995:Folland:HarmonicAnalysis}. The map $f\mapsto \widecheck{f}$, with $\widecheck{f}(x)=f(x^{-1})$, converts our definition of Fourier transform to the one in \cite{1995:Folland:HarmonicAnalysis}.

As demonstrated in the following proposition, the (matrix-valued) Fourier transform allows us to encode the non-commutative nature of non-Abelian groups.
\begin{proposition}{(Properties of Fourier transform on compact groups)}\label{prop:FT-cpt-group}
Let $\bbG$ be a compact group. For the Fourier and inverse Fourier transforms defined above we have:
\begin{itemize}
\item[(i)] (Parseval equation) For every $f\in L^2(\bbG)$, we have 
$$\|f\|_2^2=\sum_{\pi\in\widehat{\bbG}}d_\pi{\rm Tr}[\pi(f)^*\pi(f)].$$
\item[(ii)] For $f,g\in L^1(\bbG)$ and $\pi\in\widehat{\bbG}$, $\pi(f*g)=\pi(f)\pi(g).$ Here $f*g$ denotes convolution as defined in \eqref{def:convolution}.
\item[(iii)] For $f\in L^1(\bbG)$ and $\pi\in\widehat{\bbG}$, $\pi(f^*)=\pi(f)^*,$ where $f^*(x)=\overline{f(x^{-1})}$. Here, $\pi(f)^*$ denotes the usual adjoint of the matrix $\pi(f)$.
\end{itemize}
\end{proposition}
Let $\bbG$ be a compact (not necessarily Abelian) group, and $\pi$ be an irreducible representation of $\bbG$ on the Hilbert space $\bbC^{d_\pi}$.
Let $\{e_k\}_{k=1}^{d_\pi}$ be the standard basis for $\bbC^{d_\pi}$, and define the \emph{coefficient function} of $\pi$ associated with $e_i, e_j$ to be the complex-valued function on $\bbG$ defined as 
\begin{equation}\label{eq:coefficient}
\pi_{i,j}(\cdot):= \langle\pi(\cdot)e_i, e_j\rangle_{\bbC^{d_\pi}}.
\end{equation}

\begin{proposition}{(Schur's orthogonality relations \cite[Theorem 5.8]{1995:Folland:HarmonicAnalysis})}\label{prop:Schur-cpt-group}
For every $n$, let $\tU_n(\bbC)$ denote the set of all unitary matrices of size $n$. 
Let  $\pi:\bbG\to \tU_n(\bbC)$ and $\rho:\bbG\to \tU_m(\bbC)$ be inequivalent, irreducible unitary representations. Then 
\begin{itemize}
    \item[(i)] $\langle\pi_{i,j}, \rho_{r,s}\rangle_{L^2(\bbG)}=0$ for all $1\leq i,j\leq d_\pi$ and $1\leq r,s\leq d_\rho$,
    \item[(ii)] $\langle\pi_{i,j}, \pi_{r,s}\rangle_{L^2(\bbG)} = \frac{1}{d_\pi}\delta_{i,r} \delta_{j,s}$, 
\end{itemize}
where $\delta_{i,j}$ is the Kronecker delta function.
Consequently, the collection $\big\{\sqrt{d_\pi}\pi_{i,j}\big\}_{\pi\in \widehat \bbG, i,j=1,\cdots,d_\pi}$ forms an orthonormal (Hilbert) basis for $L^2(\bbG)$.
\end{proposition}

See \cite{1995:Folland:HarmonicAnalysis, Terras:1999:FourierAnalysisOnFiniteGroups} for full details of Fourier analysis of non-Abelian groups.

\subsection{Graph limit theory}\label{subsec:graphon}
Let ${\mathcal W}_0$ denote the set of all {graphons} on $[0,1]^2$, that is, the set of all  measurable functions $w: [0,1]^2 \to [0,1]$ that are symmetric, 
i.e.~$w(x, y) = w(y, x)$ for almost every point $(x, y)$ in $[0, 1]^2$. 
Let ${\mathcal W}$ denote the (real) linear span of ${\mathcal W}_0$. Every graph can be identified with a $0/1$-valued graphon as follows.

\begin{definition}\label{def:wG}
Let  $G$ be a graph on $n$ vertices labeled $ \{ 1,2,\dots ,n\}$. 
The $0/1$-valued graphon $w_G$  associated with $G$ is defined as follows: split $[0, 1]$ into $n$ equal-sized intervals $\{I_i\}_{i=1}^n$.
For every $i, j\in \{1,\ldots, n\}$, $w_G$ attains 1 on $I_i \times I_j$ precisely when vertices with labels $i$ and $j$ are adjacent. Note that $w_G$ depends on the labeling of the vertices of $G$, that is, relabeling the vertices of $G$ results in a different graphon.
\end{definition}

The topology described by convergent (dense) graph sequences can be formalized by endowing ${\mathcal W}$ with the cut-norm, introduced in \cite{cut-norm}. For  $w \in {\mathcal W}$, the cut-norm is defined as:
\begin{equation}
\label{cutnorm} 
\| w \|_{\Box}= \sup_{S,T \subset [0,1]}\left|\int_{S \times T} w(x,y)\,  dxdy \right|,
\end{equation}
 where the supremum is taken over all measurable subsets $S,T$ of $[0,1]$. \blue{Note that in the definition of the cut-norm, the sets $S$ and $T$ do not need to be of the same size.} To develop an unlabeled  graph limit theory, the cut-distance between $u,w\in{\mathcal W}$  is defined as follows.
 \begin{equation}
\label{cutdistance} 
\delta_{\Box}(u,w)= \blue{\inf}_{\sigma\in \Phi}\|u^\sigma-w\|_\Box,
\end{equation}
where $\Phi$ is the space of all measure-preserving bijections on $[0,1]$, and $w^\sigma(x,y)=w(\sigma(x),\sigma(y))$. This definition ensures that $\delta_\Box(w,u)=0$ when the graphons $w$ and $u$ are associated with the same graph $G$ given two different vertex labelings. In general, two graphons $u$ and $w$ are said to be \emph{$\delta_\Box$-equivalent} (or equivalent, for short), if $\delta_\Box(u,w)=0$.

It is known  that a graph sequence $\{ G_n\}$ converges in the sense of Lov\'{a}sz-Szegedy 
whenever the corresponding sequence of graphons $\{w_{G_n}\}$ is $\delta_\square$-Cauchy \blue{(\cite[Lemma 5.3]{borgsI2008})}. 
The limit object for such a convergent sequence can be represented as a graphon in ${\mathcal W}_0$ (not necessarily corresponding to a graph). That the graph sequence $\{G_n\}$ is convergent 
to a limit object $w \in {\mathcal W}_0$ is equivalent to $\delta_\Box(w_{G_n},w)\rightarrow 0$ as $n$ tends to infinity. This, in turn, is equivalent to the 
existence of suitable labelings for the vertices of each of the graphs $G_n$ for which we have
\begin{equation}
\label{cutdist} 
 \lVert w_{G_n} - w \rVert_{\square} = \underset{S,T \subset [0,1]} \sup \Bigl \lvert \int_{S \times T} ( w_{G_n} - w ) \Bigr \rvert \rightarrow 0.
\end{equation}
See \cite[Theorem 2.3]{BCLSV2011} for the above convergence results.

A graphon $w$ can be interpreted as a probability distribution on random graphs, sampled via the \emph{$w$-random graph} process ${\mathcal G}(n,w)$. 
The concept of {$w$-random graphs} was introduced in \cite{lovaszszegedy2006}, as a tool for generating examples of convergent graph sequences. 
For a graphon $w$, we define the random process ${\mathcal G}(n,w)$ as follows. \blue{Given the vertex set} with labels $\{ 1,2,\dots ,n\}$, edges are formed according to $w$ in two steps. First, each vertex $i$ is assigned a value $x_i$ drawn uniformly at random from $[0,1]$. Next, for each pair of vertices with labels $i<j$ independently, an edge $\{ i,j\}$ is added with probability $w(x_i,x_j)$. 
It is known that the sequence $\{{\mathcal G}(n,w)\}_n$  almost surely forms a convergent graph sequence, for which the limit object is the graphon $w$
(see \cite{lovaszszegedy2006}). 

\green{Graphons can be represented on any standard probability space $(X,\mu)$ rather than the usual choice $[0,1]$.
As given in Definition \ref{def:graphon-signal}, a function $w:X\times X\to [0,1]$ is called a graphon represented on $X$ if $w$ is measurable and symmetric. The concepts of cut-norm, cut-distance and $w$-random graphs, for a graphon $w$ on $X$, are defined analogously to the corresponding concepts for ${\mathcal W}_0$. Representing graphs and graphons on a particular probability space $X$ can offer insights on their geometric structure, which may be lost otherwise. Cayley graphons, as defined in Definition \ref{def:cayley-graphon}, provide natural examples of this phenomenon. 

There is an easy correspondence between the representation of a graphon on an atom-free standard probability space and on $[0,1]$. Let $(X,\mu)$ be an atom-free standard probability space. 
It is well-known that $X$ is isomorphic (mod 0) to the uniform probability space $[0,1]$. Let $\sigma_X$ be a fixed isomorphism (mod 0) between $[0,1]$ and $X$, that is, there are measure-zero sets $A_1\subseteq [0,1]$ and $A_2\subseteq X$ and an invertible map $\sigma_X:[0,1]\setminus A_1\to X\setminus A_2$ such that both $\sigma_X$ and its inverse are measurable and measure-preserving. Now if $w:X\times X\rightarrow [0,1]$ is a graphon on $(X,\mu)$, then 
$w_0$, defined as follows, is a representation of $w$ on $[0,1]$:
\begin{equation}\label{eq:Cayley-standard}
w_0:[0,1]^2\rightarrow [0,1], \quad w_0(x,y)=\left\{\begin{array}{cc}
   w(\sigma_X(x),\sigma_X(y))  & \mbox{ if } x,y\in [0,1]\setminus A_1, \\
   0  & \mbox{otherwise}.
\end{array}\right. 
\end{equation}
Note that the value of a graphon on a null set does not affect its graph-limit-theoretic behavior, as graphons which are almost everywhere equal belong to the same $\delta_\Box$-equivalence class. 
So, the random graphs ${\mathcal G}(n,w_0)$ and ${\mathcal G}(n,w)$ are equivalent. 

On the other hand, a (labeled) graph $G$ on the vertex set $V(G)=\{1,\ldots, n\}$ can be identified with a 0/1-valued graphon $w_{G,X}$ on $X$ as well. %
Namely, 
\begin{equation}\label{eq:W-G,X}
w_{G,X}(x,y)=\left\{\begin{array}{ll}
w_G(\sigma_X^{-1}(x),\sigma_X^{-1}(y)) & \text{for }x,y\in X\setminus A_2,\\
0 & \text{on the remaining null set,}
\end{array}\right.
\end{equation}
where $w_G\in {\mathcal W}_0$ is the graphon associated with $G$ {as in Definition~\ref{def:wG}}.
}

For a comprehensive account of dense graph limit theory, we refer the reader to \cite{lovasz-book}.

\subsection{Cayley graphons}\label{subsec:general-graphon}
In most cases, it is beneficial to represent a Cayley graphon associated with a (compact) group $\bbG$ on the probability space provided by $\bbG$. 
\begin{definition}\label{def:cayley-graphon}
Let $\bbG$ be a second countable compact group. Then $\bbG$ equipped with its Haar measure forms a standard probability space.
Let $\gamma:\bbG\rightarrow [0,1]$ be a measurable function such that $\gamma(x)=\gamma(x^{-1})$. 
Then the graphon $w:\bbG\times \bbG\to [0,1]$ defined as $w(x,y)=\gamma(xy^{-1})$ is called the Cayley graphon defined by  $\gamma$ on the group $\bbG$, and the function $\gamma$ is called a Cayley function.
\end{definition}
When $\bbG$ is a second countable, infinite, compact group, the probability space provided by $\bbG$, together with its Haar measure, is standard and atom-free. 
\blue{Therefore, we can use the approach leading to Equation~\eqref{eq:Cayley-standard} to represent a Cayley graphon on $\bbG$ as a graphon on $[0,1]^2$.}
Now suppose $\bbG$ is a finite group of size $N$, and $w_{\bbG,\gamma}:\bbG\times \bbG\rightarrow [0,1]$ is a Cayley graphon on $\bbG$ defined by the function $\gamma:\bbG\rightarrow [0,1]$. The graphon $w_{\bbG,\gamma}$ is clearly a step function, and can be represented as a step graphon $w_0$ on $[0,1]$ as follows. Split $[0, 1]$ into $N$ equal-sized intervals $\{I_s\}_{s\in\bbG}$, labeled by the elements of $\bbG$. This partition defines the map $\sigma_\bbG: [0,1]\rightarrow \bbG$ as $x\mapsto s$ if and only if $x\in I_s$. The function $\sigma_\bbG$ is a measure-preserving (not invertible) map, which allows the representation of
$w_{\bbG,\gamma}$ on $[0,1]$, defined as below:
$$w_0(x,y)=w_{\bbG,\gamma}(\sigma_\bbG(x),\sigma_\bbG(y)),$$
that is, for every $s,t\in \bbG$, the graphon $w_0$ attains the value $\gamma (st^{-1})$ on $I_s \times I_t$. 

We can sample from a Cayley graphon on a finite group via the $w$-random graph ${\mathcal G}(n,w_0)$ or, equivalently, via the $w$-random graph ${\mathcal G}(n,w_{\bbG, \gamma})$. 
To form $G\sim {\mathcal G}(n,w_{\bbG, \gamma})$, assign to each vertex with label $i\in\{1,2,\dots ,n\}$ a group element $x_i\in \bbG$ selected uniformly at random. Next, each pair of vertices with labels $i<j$ are linked independently with probability $\gamma(x_ix_j^{-1})$. The assignment of the group elements to the vertices
of $G$ can be viewed as a natural partition of the vertex set of $G$ into $|\bbG|$ subsets.
The connection probability between two vertices is then completely determined by the subsets they belong to.

\begin{remark}\label{rem:every-graphon-in-W0}
The previous paragraphs consider the cases where either the probability space is atom-free, or it is consisting entirely of atoms. In general, any graphon $w:X\times X\rightarrow [0,1]$ represented on a standard probability space $(X,\mu)$ has a $\delta_\Box$-equivalent representation $w_0$ in ${\mathcal W}_0$.  Indeed, every standard probability space is isomorphic mod 0 to a disjoint union of a closed interval (equipped with the Lebesgue measure) and a countable set of atoms. A combination of the above arguments for finite and infinite compact groups can be easily adjusted to produce a (not necessarily injective) measure-preserving map $\sigma_X:[0,1]\to X$. 
The representation $w_0\in {\mathcal W}_0$ is then defined as $w_0(x,y)=w(\sigma_X(x),\sigma_X(y))$.
\end{remark}

\subsection{Spectral decomposition of graphons}\label{sec:SpectralDecompGraphons}
For the rest of this article, we consider graphons in their general form, i.e.~represented on a standard measure space $(X,\mu)$. 
We may assume that $X$ is infinite; as seen in the previous section, graphons derived from finite graphs or groups can be represented as graphons on the infinite probability space $[0,1]$.
Every graphon $w:X\times X\to [0,1]$ can act as the kernel of an integral operator on the Hilbert space $L^2(X)$ as follows:
$$T_w: L^2(X)\to L^2(X), \quad T_w(\xi)(x)=\int_X w(x,y)\xi(y)\, d\mu(y), \mbox{ for } \xi\in L^2(X), x\in X.$$ 
By removing a null set if necessary, we can assume \emph{wlog} that $X$ is chosen so that $L^2(X)$ is separable. In particular, if $X$ is a locally compact group, we assume it is second countable. This assumption guarantees the existence of a finite or countable orthonormal basis for $L^2(X)$.  

Since $w\in L^2(X\times X)$ and as $w$ is real-valued and symmetric, the  Hilbert-Schmidt operator $T_w$ is self-adjoint. In addition, the operator norm of $T_w$ is $\|w\|_\infty$, which is bounded by 1. Thus, $T_w$ has a countable spectrum lying in the interval $[-1,1]$ for which 0 is the only possible accumulation point. We label the nonzero eigenvalues of $T_w$ as follows:
\begin{eqnarray}\label{eq:ordering}
1\geq \lambda_1(w)\geq \lambda_2(w)\geq \ldots\geq 0\geq \ldots \geq \lambda_{-2}(w)\geq \lambda_{-1}(w)\geq -1
\end{eqnarray}
Note that the set of positive eigenvalues of $T_w$ may be finite. In that case, we pad the sequence with 0's at the end, so we can view it as an infinite sequence. {We do} similarly for the set of negative eigenvalues. This arrangement is important for our discussion of the convergence of spectra for converging sequences of dense graphs (e.g.~in Theorem~\ref{thm:spectrum}).

Using spectral theory for compact operators, we see that $L^2(X)$ admits an orthonormal basis containing eigenvectors of $T_w$; this results in a spectral decomposition for $T_w$. More precisely, let $I_w\subseteq {\mathbb Z}^*$ be the indices in \eqref{eq:ordering} enumerating nonzero eigenvalues of $T_w$, and let $\{\phi_i\}_{i\in I_w}$ be an orthonormal collection of associated eigenvectors. Then we have
\begin{equation}\label{eq:spectral-decom}
T_w=\sum_{i\in I_w}\lambda_i(w) \, \phi_i\otimes \phi_i,
\end{equation}
where $\phi_i\otimes \phi_i$ denotes the rank-one projection on $L^2(X)$ defined as $(\phi_i\otimes\phi_i)(\xi)=\langle\xi,\phi_i\rangle\phi_i$ for every $\xi\in L^2(X)$. Note that the infinite sum in the above spectral decomposition should be interpreted as operator-norm convergence in the space of bounded operators on $L^2(X)$. Since $T_w$ is a Hilbert-Schmidt operator, the spectral decomposition sum converges in the Hilbert-Schmidt norm as well. In particular, given the spectral decomposition (\ref{eq:spectral-decom}), we have 
\begin{equation}\label{eq:spectral-l2}
w=\sum_{i\in {I}_w}\lambda_i(w) \, \phi_i\otimes \phi_i,
\end{equation}
where $\phi_i\otimes \phi_i\in L^2(X\times X)$ is defined to be $(\phi_i\otimes \phi_i)(x,y)=\phi_i(x)\phi_i(y)$, and the convergence of the infinite sum is interpreted as convergence in $L^2(X\times X)$. 
Note that by slight abuse of notation, we use $\phi_i\otimes \phi_i$ to denote a rank-one projection on $L^2(X)$ or its associated kernel in $L^2(X\times X)$.

For background material on compact operators and their spectral theory, we refer to \cite[Chapters 13--14]{Bollobas-book}. 

We finish the discussion about spectra by quoting two theorems on convergence of spectra.
\begin{theorem}{\cite[Theorem 11.54]{lovasz-book}}\label{thm:spectrum}
Let $\{w_n\}_{n\in\bbN}$ be a sequence of graphons converging to a graphon $w$ in $\delta_\Box$-distance. Then for every fixed $i\in\bbZ^*$, we have %
$$\lambda_i(w_n)\rightarrow \lambda_i(w)\ \mbox{ as } \ n\rightarrow \infty,$$
{where the eigenvalues of each graphon are indexed as in \eqref{eq:ordering}.}
\end{theorem}
A more careful analysis of convergence of spectra was given in \cite{Szegedy-spectra} where the convergence of eigenspaces was shown. 
For a graphon $w$ with spectral decomposition as in \eqref{eq:spectral-decom} and a positive number $\alpha>0$, define
\begin{equation}\label{eq:cut-off}
[w]_\alpha=\sum_{\{i\in {\mathbb Z}:\ |\lambda_i(w)|>\alpha\}}\lambda_i(w) \, \phi_i\otimes \phi_i.
\end{equation}
\begin{theorem}{\cite[Proposition 1.1]{Szegedy-spectra}}\label{thm:szegedy-spectra}
Let $\{w_n\}_{n\in\bbN}$ be a sequence of graphons.  Then the following two statements are equivalent:
\begin{itemize}
\item[(i)]  The sequence $\{w_n\}_{n\in\bbN}$ converges to $w$ in cut-norm.
\item[(ii)] There is a decreasing positive real sequence $\{\alpha_n\}_{n\in\bbN}$ approaching to 0 such that for every $j\in\bbN$, we have 
$\|[{w_n}]_{\alpha_j}-[w]_{\alpha_j}\|_{L^2(X\times X)}\rightarrow 0$.
\end{itemize}
Furthermore, in the second statement the cut-norm limit $w$ of $\{w_n\}_{n\in\bbN}$ can be computed as 
$$w=\lim_{j\rightarrow \infty}\left(\lim_{i\rightarrow \infty}[w_i]_{\alpha_j}\right)$$
converging in $L^2(X\times X)$. 
\end{theorem}


\section{Graphon signal processing}\label{sec:generalize}
Graphons can be viewed as limit objects of graph sequences. Consequently, it is natural to develop the idea of a Fourier transform on graphons in such a way that the graph Fourier transform, along a converging graph sequence, converges (in some appropriate topology) to the graphon Fourier transform of the limit object. Such an approach was first proposed by Ruiz, Chamon and Ribeiro in \cite{ruiz2}, and expanded upon in \cite{RuizChamonRibeiro21}. They define a graphon Fourier transform based on the spectral decomposition of the graphon, and give a convergence result restricted to the class of so-called non-derogatory graphons. In this section, we show that the restriction can only be removed with a broader definition of the graphon Fourier transform, which is independent of the choice of basis for each eigenspace. Using this new definition, we establish a more general convergence result stated in Theorem \ref{thm:convergence}.  

The graphon Fourier transform as defined in \cite{RuizChamonRibeiro21} is evidently motivated by Fourier analysis on ${\mathbb R}$.
Namely, it is derived from an orthonormal basis of $L^2(X)$, say ${\mathcal B}$, consisting of eigenvectors of $T_w$. 
The {graphon Fourier transform (WFT)} of a graphon signal $(w,f)$ is then defined via expansion of $f$ with respect to the basis ${\mathcal B}$, that is,
\begin{equation*}\label{WFT}
\widehat{f}(\phi)=\langle f,\phi\rangle=\int_X f(x)\overline{\phi(x)}\,d\mu(x) \ \mbox { for } \phi\in {\mathcal B},\mbox{ (as defined in \cite{RuizChamonRibeiro21}}).
\end{equation*}
The inverse graphon Fourier transform (iWFT)  of $\widehat{f}$ is given by 
\begin{equation*}\label{iWFT}
{\rm iWFT}(\widehat{f})=\sum_{\phi\in {\mathcal B}}\widehat{f}(\phi)\phi, \mbox{ (as defined in \cite{RuizChamonRibeiro21}}).
\end{equation*}
Since ${\mathcal B}$ is an orthonormal basis of $L^2(X)$, we have ${\rm iWFT}(\widehat{f})=f$, with the equality interpreted in $L^2(X)$.

We now describe precisely what we mean by convergence of a sequence of graph/graphon signals to a graphon signal, using the framework discussed in Section \ref{subsec:general-graphon}. 
Let $G$ be a graph on $n$ vertices labeled $\{1,2,\dots,n\}$, and let $w:X\times X\rightarrow [0,1]$ be a graphon represented on an infinite standard probability space $X$.
With the given labeling for the vertex set of $G$, a signal on $G$ is just a function $f:\{ 1,2,\dots,n\}\rightarrow \bbC$.
With respect to this labeling,  every graph signal can also be viewed as a step function in $L^2[0,1]$ by identifying each vertex $v$ labeled $i$ with the interval $I_i=[\frac{i-1}{n},\frac{i}{n})$. 
%
By Remark \ref{rem:every-graphon-in-W0}, the graphon $w$ can be transformed to a graphon $w_0 \in {\W}_0$ using a measure-preserving map $\sigma_X:[0,1]\to X$. 
Applying the same measure-preserving map to the step function $f\in L^2[0,1]$ allows us to transform $f$ to a signal $f^X\in L^2(X)$ on the graphon $w$; namely, we define
\begin{equation}
    \label{eqn:fX}
f^X(s)=\sqrt{n}f(k) \quad \forall s\in \sigma_X(I_k),\, 1\leq k\leq n.
\end{equation}

Note that the scaling factor of $\sqrt{n}$ in \eqref{eqn:fX} ensures that the map $f\mapsto f^X$ is an isometry from $\bbC^n$ to $L^2(X)$:
if $f,g:\{ 1,2,\dots ,n\}\to \bbC$ are signals on a graph $G$ with $n$ vertices, then
\begin{equation}
\label{eq:square-root-factor}
\langle f^X,g^X\rangle_{L^2(X)} = \langle f,g\rangle_{\bbC^n}.
\end{equation}
Consequently, applying the graph shift operator on $f$ in $\bbC^n$ yields the same result as applying the corresponding graphon shift operator to the functions $f^X$ in $L^2(X)$. Namely, let $G$ be a graph on $n$ vertices labeled as above. Let $A$ be the adjacency matrix of $G$, and $f$ be a signal on $G$ viewed as a vector $f\in \bbC^n$. Fix an infinite standard probability space $X$, with the measure-preserving map $\sigma_X:[0,1]\to X$.  Let $w_{G,X}:X\times X \rightarrow [0,1]$ be the graphon associated with $G$ represented on $X$ (as defined in \eqref{eq:W-G,X}). Then
\begin{equation}\label{eq:Tdiscrete-vs-cts}
T_{w_{G,X}}f^X (x)= \int_{X} w_{G,X}(x,y)f^X(y)\, dy=\sum_{j=1}^n \frac{1}{\sqrt{n}}A_{k,j}f_j = (Af)^X(x), \mbox{ for } x\in \sigma_X(I_k).
\end{equation}

To discuss convergence of graphon signals, we need to clearly distinguish between convergence in cut-norm \blue{of graphon signals (Definition~\ref{def:graph_signal_convergence-norm})}
and convergence in cut-distance \blue{of (unlabeled) graph signals (Definition~\ref{def:graph_signal_convergence})}. 
\begin{definition}\label{def:graph_signal_convergence-norm}
We say a sequence $\{(w_n,f_n)\}_{n\in\bbN}$ of graphon signals on a standard probability space $(X,\mu)$ converges \emph{in norm} to a graphon signal $(w,f)$ 
if $$\|w_n-w\|_\Box\rightarrow 0,
\ \mbox{ and } \
\|{f}_n-f\|_2\rightarrow 0.$$
\end{definition}
%
%
\green{
\begin{definition}\label{def:graph_signal_convergence}
Fix an infinite standard probability space $(X,\mu)$, together with a measure-preserving map $\sigma_X:[0,1]\to X$.
A sequence $ \{(G_n, f_n)\}_{n\in \bbN}$ of graph signals converges to a graphon signal $(w,f)$ represented on $(X,\mu)$ 
if there exist labelings of each of the graphs $G_n$ so that the suitably labeled graphon signal sequence $\{(w_{G_n, X},f^X_n)\}_{n\in\bbN}$ converges in norm. 
\end{definition}}
Ruiz \emph{et al.}~prove a convergence result of the GFT of  graph signals to the WFT of the limiting graphon. Their result is limited to graphons and signals with certain properties. We restate the convergence result here; see \cite[Theorem 1]{RuizChamonRibeiro21} for the original statement.
\begin{theorem*}{\rm \cite[Theorem 1]{RuizChamonRibeiro21}}
Let $\{(G_n , f_n )\}$ be a sequence of graph signals converging to the graphon signal $(w,f)$. 
Assume all graphs are labeled to ensure convergence as in Definition \ref{def:graph_signal_convergence}. 
Suppose the following conditions hold:
\begin{itemize}
    \item[(i)] The signal $(w,f)$ is $c$-bandlimited for some $c>0$. That is, $\widehat{f}(\chi)=0$  whenever $\chi$ is a $\lambda$-eigenvector of $T_w$ with $|\lambda|<c$.
    \item[(ii)] The graphon $w$ is non-derogatory, i.e.~every eigenvalue of $T_w$ has multiplicity 1. 
\end{itemize}
Let $\{\phi_j\}$ and $\{\phi_j^{n}\}$ denote normalized eigenvectors associated with nonzero eigenvalues of $w$ and $G_n$ respectively, ordered as in \eqref{eq:ordering}.
Then, 
we have that
$\{{\rm GFT}( G_n, f_n)\} $ converges to ${\rm WFT}(w,f)$ in the sense that for every index $j$, we have 
$$\frac{1}{\sqrt{|G_n|}}\widehat{f_n}(\phi_j^{n}) \rightarrow \widehat{f}(\phi_j) \mbox{ as } n\rightarrow \infty.$$
\end{theorem*}

The condition that convergence only holds for non-derogatory graphons is highly restrictive. While non-derogatory graphons form a dense subset in the space of all graphons, the above theorem does not imply continuity of signal processing on the whole space. Moreover, the restriction excludes classes of graphons that have proven useful in practice, such as those of \emph{stochastic block models} \cite{Abbe}. Another such example is the class of Cayley graphons, which provide a versatile tool to model graphs whose link structure is informed by an underlying group and its topology. As it turns out, Cayley graphons tend to have many nonzero eigenvalues of multiplicity higher than 1. 

To extend the above convergence result to all graphons, we need to modify the definition of the graphon Fourier transform.  
Instead of defining the WFT as the projection of a signal on each element of the eigenbasis of the graphon, we think of the projection onto each eigenspace of $w$.
The two definitions coincide when $w$ is non-derogatory. However, the latter approach enables us to handle eigenvalues of higher multiplicity, as it provides us with a definition which is independent of the particular choice of eigenbasis. 

To precisely state the new definition of graphon Fourier transform, and to analyze convergence of the graphon Fourier transform along a converging graphon sequence in this context, we need the following notation:

\begin{notation}\label{notation:notation-thm1}
Let $w$ be a (not necessarily non-derogatory) graphon on an infinite standard probability space $(X,\mu)$. 
Let $w=\sum_{i\in I_w}\lambda_i(w) \, \phi_i\otimes \phi_i$ be the spectral decomposition of $w$ as in \eqref{eq:spectral-l2}, where $\{\lambda_i(w)\}_{i\in {I_w}}$ are {nonzero} eigenvalues of the associated integral operator $T_w$, and $\{\phi_n\}_{n\in{I_w}}$ is an orthonormal set of eigenvectors of $T_w$, associated with nonzero eigenvalues. Thus, $I_w\subseteq {\mathbb Z}^*$ is the set of indices $j$ such that $\lambda_j(w)\neq 0$. Let  $\{\mu_j(w)\}_{j\in \bbZ^*}$ be the sequence of  \emph{distinct} nonzero eigenvalues of $T_w$. The sequence is padded with zeros if the number of positive or negative eigenvalues is finite.  We always order eigenvalues as in (\ref{eq:ordering}) .

For each $\mu_j(w)$, let 
$$I_{\mu_j}=\{i\in I_w: \ \lambda_i(w)=\mu_j(w)\}.$$
By definition of $I_w$, if $\mu_j= 0$ then $I_{\mu_j}=\emptyset$.

For a subset $I\subseteq {I_w}$, define the operator $P^w_{I}: L^2(X)\to L^2(X)$ as
$$P^w_{I}=\sum_{i\in I}\phi_i\otimes \phi_i.$$
Clearly this is an orthogonal projection.
We set $P^w_{I}$ to be the zero operator when $I=\emptyset$.
For each $\mu_j(w)\neq 0$, the operator $P^w_{I_{\mu_j}}$ is the orthogonal projection onto the $\mu_j(w)$-eigenspace of $T_w$, and is of finite rank. 

Finally, define $P^w_0$ to be the orthogonal projection onto the null space of $T_w$. Contrary to the previous projections, $P^w_0$ is not
necessarily of finite rank.
\end{notation}
With this notation, we can now formalize the new definition of the graphon Fourier transform  as follows.

\begin{definition}[Graphon Fourier Transform]\label{def:NC-Fourier}
Let $w:X\times X\rightarrow [0,1]$ be a graphon, and $\Sigma$ denote the set of distinct eigenvalues of $T_w$.  The graphon Fourier transform of a graphon signal $(w,f)$ is a vector-valued function $\widehat{f}$ on $\Sigma$  defined as 
\begin{equation}\label{eqn:FT_Proj}
\widehat{f}(\mu_j)= P^w_{I_{\mu_j}}(f)\mbox{ for every nonzero } \mu_j, \text{ and } \widehat{f}(0) = P_0^w(f),
\end{equation}
where the notation is as given in (\ref{notation:notation-thm1}).

The \emph{inverse Fourier transform} can then be expressed as an infinite sum in $L^2(X)$:
\begin{equation}\label{eqn:iWFT}
f= \sum_{j\in\bbZ^*}P^w_{I_{\mu_j}}(f) +P_0^w(f) =\sum_{j\in \bbZ^*} \widehat{f}(\mu_j)+\widehat{f}(0).
\end{equation}
\end{definition}
The Fourier transform defined above exhibits principal features expected from a graphon Fourier transform. Most importantly, this new definition allows appropriate convergence behavior of the Fourier transform, when applied to any convergent sequence of graphs (with no restriction on the limiting graphon) as we will see in Theorem~\ref{thm:convergence}. 
Such general convergence results can only be achieved by paying a price: each graphon Fourier transform `coefficient' is defined as a vector--often lying in an infinite dimensional space--rather than a numerical value. We do not view this fact as a drawback of our approach/definition; indeed, this vector-valued definition is in line with the harmonic analytic definition of the Fourier transform when the ambient group is non-Abelian.

\begin{remark}{(Conventions for graph and graphon Fourier transforms.)} \label{remark:GFT-WFT}
Throughout this paper, graph signals are considered as vectors in $\bbC^n$, and the graph Fourier transform is as defined in \eqref{GFT}. 
The graphon Fourier transform is, however, the projection as defined in Definition \ref{def:NC-Fourier}. 
\end{remark}
For a graph $G$ and an associated graphon $w_G$, the graph Fourier transform on $G$ and the graphon Fourier transform on $w_G$ are closely related, as demonstrated in the following lemma.
\begin{lemma}\label{lem:graph-graphon-spec}
Let $X$ be an infinite standard probability space with a measure-preserving map $\sigma_X:[0,1]\to X$, let $G$ be a graph on $n$ vertices, and let $f\in \bbC^n$ be a signal on $G$.  
Let $w_{G,X}$ be the graphon associated with $G$ and represented on $X$ (as defined in \eqref{eq:W-G,X}), and consider the graphon signal $(f^X, w_{G,X})$ 
associated with the graph signal $(f,G)$ as defined in \eqref{eqn:fX}. Then we have:
\begin{itemize}
\item[(i)] Every eigenvector of a nonzero eigenvalue $\lambda$ of  $T_{w_{G,X}}$ is of the form $\phi^X$ for some $\phi\in \bbC^n$. 
\item[(ii)] Let $\lambda\neq 0$. Then $\lambda$ is an eigenvalue of $A_G$ of multiplicity $m$ with $\lambda$-eigenbasis $\{\phi_1,\ldots,\phi_m \}\subseteq \bbC^n$ \emph{iff} 
$\lambda$ is an eigenvalue of $T_{w_{G,X}}$ of multiplicity $m$ with $\lambda$-eigenbasis $\{\phi_1^X,\ldots,\phi_m^X \}\subseteq L^2(X)$.
Moreover,  
$$\widehat{f^X}(\lambda)=\sum_{i=1}^m\widehat{f}(\phi_i)\phi_i^X,$$
\blue{where $\widehat{f^X}(\lambda)$ is the graphon Fourier transform as defined in Definition \ref{def:NC-Fourier} and $\widehat{f}(\phi_i)$ is the graph Fourier transform as given in \eqref{GFT}.}
\end{itemize}
\end{lemma}
\begin{proof}
Let $I_1,\ldots, I_{n}$ denote the partition of $[0,1]$ into $n$ equal-sized intervals. 
Let $h\in L^2(X)$, and $1\leq i\leq n$. For $x\in \sigma_X(I_i)$, we have 
\begin{equation*}
T_{w_{G,X}} h (x)=\int_X w_{G,X}(x,y) h(y)\, dy= \sum_{j=1}^n\int_{\sigma_X(I_j)} w_{G,X}(x,y) h(y)\, dy= \sum_{j=1}^n A_G(i,j)\int_{\sigma_X(I_j)} h(y)\, dy.
\end{equation*}
So, the function $T_{w_{G,X}} h$ is constant on each set $\sigma_X(I_i)$, when $1\leq i \leq n.$
Consequently, every eigenvector of $T_{w_{G,X}}$ associated with a nonzero eigenvalue must attain constant values on the same subsets. This finishes the proof of (i).

The correspondence between nonzero eigenvalues/eigenvectors of $A_G$ and $T_{w_{G,X}}$ follows from (i) together with \eqref{eq:Tdiscrete-vs-cts}.
Finally, note that $\widehat{f}(\phi_i)=\langle f, \phi_i\rangle_{\bbC^n}=\langle f^X, \phi_i^X\rangle_{L^2(X)}$. This finishes the proof of (ii).
\end{proof}

Theorem \ref{thm:convergence} below is stated in the most general form, dealing with sequences of graphon signals that converge in norm. In Corollary \ref{cor:GraphConvergence}, we will apply the theorem to converging graph signals.
%
\begin{theorem}\label{thm:convergence}
Let $w:X\times X\rightarrow [0,1]$ be a graphon with terminology as in Notation \ref{notation:notation-thm1}. 
Let  $\{(w_n , f_n)\}$ be a sequence of graphon signals, all represented on $X$, converging in norm to a graphon signal $(w, f)$. 
Then, for every nonzero $\mu_j$, 
$P^{w_n}_{I_{\mu_j}}\rightarrow P^{w}_{I_{\mu_j}}$ in Hilbert-Schmidt norm. In particular, we have
\begin{equation}\label{eq:converg}
P^{w_n}_{I_{\mu_j}}(f_n) \rightarrow P^w_{I_{\mu_j}}(f) \mbox{ in } L^2(X) \mbox{ as }  n\rightarrow \infty.
\end{equation}
Moreover, if  $P^w_0(f)=\mathbf{0}$ (the zero function) then we have that 
\begin{equation}\label{eq:converg2}
\sum_{j\in \bbZ^*} \| P^{w_n}_{I_{\mu_j}}(f_n)- P^{w}_{I_{\mu_j}}(f)\|^2_2
\rightarrow 0  \mbox{ as } n\rightarrow \infty.
\end{equation}
\end{theorem}
\begin{remark}{(Remarks about the proof.)}
As defined earlier, $T_w$ is the integral operator associated with the (Hilbert-Schmidt) kernel $w\in L^2(X\times X)$. It is well-known that the Hilbert-Schmidt norm of $T_w$, denoted by $\|T_w\|_2$, and the $L^2$-norm of $w$ are equal. 
In the following proof, we use $\phi\otimes \phi$ to denote both the rank-one operator on $L^2(X)$ defined as $f\mapsto \langle f, \phi\rangle \phi$   and the  kernel of that operator which is the function in $L^2(X\times X)$ given by $(\phi\otimes \phi)(x,y)=\phi(x)\phi(y)$. Notations such as $\|\phi\otimes\phi\|_2$ can be interpreted both as norm of a function in $L^2(X\times X)$ or the Hilbert-Schmidt norm of the associated integral operator.
\end{remark}
\begin{proof}
We apply Notation \ref{notation:notation-thm1} to all graphons in the sequence $\{w_n\}$. That is, $\{ \lambda_i(w_n)\}_{i\in I_{w_n}}$ is the sequence of (repeated) eigenvalues of $w_n$ ordered as in \eqref{eq:ordering}, and $\{\phi^n_i\}$ is the sequence of associated eigenvectors. 
Thus 
$$
w_n=\sum_{i\in I_{w_n}}\lambda_i(w_n)\phi^n_i\otimes \phi^n_i
$$ 
is the spectral decomposition of $w_n$. 

Next, fix a decreasing sequence of positive numbers $\{\alpha_i\}_{i\in\bbN}$ converging to 0, and assume that the sequences $\{\alpha_i\}_{i\in\bbN}$ and $\{|\mu_j|\}_{j\in\bbZ^*,\mu_j\neq 0}$ are interlacing sequences with no common terms. 
Then, $\{\alpha_i\}_{i\in\bbN}$ satisfies the condition of  Theorem~\ref{thm:szegedy-spectra} (ii), so
$\|[{w_n}]_{\alpha_i}-[w]_{\alpha_i}\|_{L^2(X\times X)}\to 0$ for each $i\in{\mathbb N}$. 

Fix $j\in \bbZ^*$, and suppose $\mu_j(w)>0$; 
the case where $\mu_j(w)$ is negative can be done in an identical manner. 
We choose $r_j=\alpha_i$ and  $s_j=\alpha_{i+1}$ so that $(s_j,r_j)\cap {\{ |\mu_i|:i\in \bbZ^*\}}=\{\mu_j(w)\}$ is a singleton, so we have 
\begin{equation}\label{eq:Pdef}
[w]_{s_j}-[w]_{r_j}=\sum_{i\in I_{\mu_j}}\mu_j(w)\phi_i\otimes \phi_i-\sum_{i\in I_{-\mu_j}}\mu_j(w)\phi_i\otimes \phi_i.
\end{equation}
Let $P^+_j$ (resp.~$P^-_j$) denote the orthogonal projection onto the $\mu_j(w)$-eigenspace (resp.~the eigenspace associated with $-\mu_j(w)$).
Then we can rewrite (\ref{eq:Pdef}) as
\begin{equation}\label{eq:Pdef2}
[w]_{s_j}-[w]_{r_j}=\mu_j(w)(P^+_j-P^-_j).
\end{equation}
We now invoke Theorem \ref{thm:szegedy-spectra} to obtain that
\begin{eqnarray}
\|([{w_n}]_{s_j}-[{w_n}]_{r_j})-([w]_{s_j}-[w]_{r_j})\|_2&=&\|([{w_n}]_{s_j}-[{w}]_{s_j})+([w]_{r_j}-[w_n]_{r_j})\|_2\nonumber\\
&\leq&\|[{w_n}]_{s_j}-[{w}]_{s_j}\|_2+\|[w]_{r_j}-[w_n]_{r_j}\|_2\label{eq:conv}\\
&\to& 0 \quad \mbox{ as } n\rightarrow \infty.\nonumber
\end{eqnarray}

For each $n\in \bbN$, define the finite-rank projections 
\begin{equation*}
P_{n,j}^+:=\sum_{i\in I_{\mu_j(w)}}  \phi_i^n\otimes \phi_i^n\ \mbox{ and } \  P_{n,j}^-:=\sum_{i\in I_{-\mu_j(w)}} \phi_i^n\otimes \phi_i^n.
\end{equation*}
From Theorem \ref{thm:spectrum},  $\lim_{n}\lambda_i(w_n)=\mu_j(w)$ if $i\in I_{\mu_j(w)}$, and $\lim_{n}\lambda_i(w_n)=-\mu_j(w)$  if $i\in I_{-\mu_j(w)}$. For all other  $i$, the value $\lim_{n}\lambda_i(w_n)$ does not fall in $(s_j,r_j)$. Thus for large values of $n$, 
$\lambda_i(w_n)\in (s_j,r_j)$ \emph{iff} $i\in I_\mu(w)$. Moreover, as $n$ approaches infinity, $\lambda_i(w_n)\to \mu_j(w)$ for $i\in I_\mu(w).$ A similar statement holds for $I_{-\mu(w)}$.
Therefore, we have 
\begin{equation}\label{eq-conv2}
\|\big([{w_n}]_{s_j}-[{w_n}]_{r_j}\big)-\mu_j(w)\left(P_{n,j}^+-P_{n,j}^-\right)\|_2\rightarrow 0 \mbox{ as } n\to \infty.
\end{equation}
Putting \eqref{eq:Pdef2}, \eqref{eq:conv}, and \eqref{eq-conv2} together, and using the fact that $\mu_j(w)\neq 0$, we get that 
$$\left\|\left(P^+_j-P^-_j\right)-\left(P_{n,j}^+- P_{n,j}^-\right)\right\|_2\rightarrow 0 \mbox{ as } n\to \infty.$$

Let ${\mathcal B}(L^2(X))$ denote the space of bounded linear operators on $L^2(X)$ equipped with operator norm. 
Since Hilbert-Schmidt norm dominates the operator norm, we have 
\begin{equation}\label{eq:easy1}
P_{n,j}^+- P_{n,j}^-\to P^+_j-P^-_j  \ \mbox{ in } {\mathcal B}(L^2(X)).
\end{equation}
So,  $(P_{n,j}^+- P_{n,j}^-)^2\to (P^+_j-P^-_j)^2$ as well in ${\mathcal B}(L^2(X))$. Note that $P^+_jP^-_j=P^-_jP^+_j=0$ and $P_{n,j}^+P_{n,j}^-=0=P_{n,j}^-P_{n,j}^+$, since they are orthogonal projections and the images of each pair are orthogonal subspaces of $L^2(X)$. 
Applying this, together with the fact that every projection is an idempotent, we obtain

\begin{equation}\label{eq:easy2}
P_{n,j}^++ P_{n,j}^-\to P^+_j+P^-_j \ \mbox{ in } {\mathcal B}(L^2(X)). 
\end{equation}
Adding and subtracting \eqref{eq:easy1} and \eqref{eq:easy2} imply that $P_{n,j}^+\to P^+_j$ and $P_{n,j}^-\to P^-_j$ in ${\mathcal B}(L^2(X))$ as $n\rightarrow \infty$.  
Moreover, note that the operators $P_{n,j}^+$, $P^+_j, P_{n,j}^-$ and $P^-_j$ are Hilbert-Schmidt operators, and $\|P_{n,j}^+\|_2, \|P^+_j\|_2, \|P_{n,j}^-\|_2, \|P^-_j\|_2\leq \max\{|I_{\mu_j(w)}|, | I_{-\mu_j(w)}|\}$. Using this uniform bound, we can now prove convergence in the Hilbert-Schmidt norm as follows:
\begin{eqnarray*}
\|P_{n,j}^+-P^+_j\|_2&=&\|(P_{n,j}^+-P^+_j)(P_{n,j}^++P^+_j)+P^+_j(P_{n,j}^+ -P^+_j)+(P^+_j-P_{n,j}^+)P^+_j\|_2\\
&\leq&\|P_{n,j}^++P^+_j\|_2\|P_{n,j}^+-P^+_j\|_{{\mathcal B}(L^2(X))}+\|P^+_j\|_2\|P_{n,j}^+-P^+_j\|_{{\mathcal B}(L^2(X))}\\
&\qquad\qquad +&
\|P^+_j\|_2\|P^+_j-P_{n,j}^+\|_{{\mathcal B}(L^2(X))},
\end{eqnarray*}
which converges to 0 as $n$ tends to infinity. Here, we have used the fact that Hilbert-Schmidt operators form an ideal in ${\mathcal B}(L^2(X))$.
Namely, if $T\in {\mathcal B}({\mathcal H})$ and $S$ is a Hilbert–Schmidt operator on the Hilbert space ${\mathcal H}$ then $\|TS\|_{2}\leq \|T\|_{{\mathcal B}({\mathcal H})}\|S\|_{2}$ and $\|ST\|_{2}\leq \|T\|_{{\mathcal B}({\mathcal H})}\|S\|_{2}$. 
This completes the first part of the theorem. 

To prove the second part, fix a vector $f\in L^2(X)$, and assume that $P^w_0(f)=\mathbf{0}$. We will show that 
\begin{equation}\label{eqn:wft}
\sum_{j\in \bbZ^*} \| P^{w_n}_{I_{\mu_j}}(f)- P^{w}_{I_{\mu_j}}(f)\|^2_2
\rightarrow 0
\text{ as } n\rightarrow \infty.
\end{equation}
This suffices to prove the claim of the theorem since $\sum_{j\in \bbZ^*} \| P^{w_n}_{I_{\mu_j}}(f-f_n)\|^2_2\leq \| f-f_n\|^2_2$, and $f_n\rightarrow f$  in $L^2(X)$ as $n\rightarrow \infty$, by definition. 

\emph{Wlog} assume that $f\neq \mathbf{0}$, as \eqref{eqn:wft} trivially holds if $f=\mathbf{0}$.
To simplify notation, let $P_{n,j}=P^{w_n}_{I_{\mu_j}}$ and $P_j=P^{w}_{I_{\mu_j}}$.
Recall that $\{P_j\}_j$ (resp.~$\{P_{n,j}\}_j$ for each $n\in {\mathbb N}$) is a collection of pairwise orthogonal projections.  
To prove (\ref{eqn:wft}), let $\epsilon>0$ be given.
The collection $\{\phi_i\}_{i\in I_w}$, together with any orthonormal basis of the null space, forms an orthonormal basis for $L^2(X)$. Thus, using the fact that $P^w_0(f)=0$, we can  decompose $f$ into orthogonal components
$f=\sum_{j\in\bbZ^*}P_{j}(f)$. 
(Recall that $P_j$ is defined to be the zero operator if $\mu_j=0$, and thus $I_{\mu_j}=\emptyset.$)
Consequently,  we have 
\[\|f\|_2^2=\sum_{j\in\bbZ^*}\|P_j(f)\|_2^2.\]
Since the above sum is bounded, there exists a finite set $S\subset I_w$ such that 
\begin{equation}\label{eq:cut-f}
    \left\|f-\sum_{j\in S}P_j(f)\right\|_2<\frac{\epsilon}{4}. 
\end{equation}
Let $h:=\sum_{j\in S}P_j(f)$, and note that $\| f-h\|_2<\epsilon/4$.

From the first part of the theorem, we have that, for each $j\in S$,
$
\left\| P_{n,j}(h)- P_{j}(h)\right\|^2_2\rightarrow 0, 
$
as $n\rightarrow \infty$.
Given that $S$ is finite, there must exist $N\in \bbN$ so that for all $n\geq N$,
\begin{equation*}
 \sqrt{\sum_{j\in S} \| P_{n,j}(h)-P_j(h) \|_2^2} <\min\left\{\frac{\epsilon}{4},\frac{\epsilon^2}{32\|f\|_2}\right\}.
\end{equation*}
To show (\ref{eqn:wft}), we use the triangle inequality in the space $\ell^2\text{-}\oplus_{j\in\bbZ^*}L^2(X)$.
Observe that
\begin{eqnarray}
\sqrt{\sum_{j\in \bbZ^*}\| P_{n,j} (f)-P_j (f)\|^2_2} &\leq & \sqrt{\sum_{j\in\bbZ^*} \| P_{n,j}(f-h)\|^2_2} + \sqrt{\sum_{j\in\bbZ^*} \| P_{j}(f-h)\|^2_2} \notag\\
&& \quad\quad + \sqrt{\sum_{j\in\bbZ^*} \| P_{n,j}(h)-P_j(h)\|^2_2}\notag \\
&\leq &2\| f-h\|_2 + \sqrt{\sum_{j\in S} \| P_{n,j}(h)-P_j(h)\|^2_2}  \notag\\
&&\quad \quad + \sqrt{\sum_{j\in \bbZ^*\setminus S}  \| P_{n,j}(h)\|^2_2},\notag
\end{eqnarray}
where the last step follows since, for any $j\in \bbZ^*\setminus S$, $P_j(h)=\mathbf{0}$. 

It remains to show that $\sqrt{\sum_{j\in \bbZ^*\setminus S}  \| P_{n,j}(h)\|^2_2}<\epsilon/4$.
For $n\geq N$, using the triangle inequality, we have
\[
\sqrt{\sum_{j\in S}  \| P_{n,j}(h)\|^2_2} \,\geq \,  \sqrt{\sum_{j\in S}  \| P_{j}(h) \|^2_2} - \sqrt{\sum_{j\in S} \| P_{n,j}(h)-P_j(h)\|^2_2} \,\geq \, \| h\|_2-\frac{\epsilon^2}{32\|f\|_2}.
\]
Since $\sqrt{\sum_{j\in \bbZ^*} \| P_{n,j}(h)\|^2_2} \leq \| h\|_2$, the above inequality implies that 
\begin{eqnarray}\label{eq1-estimate}
&&\sqrt{\sum_{j\in \bbZ^*} \| P_{n,j}(h)\|^2_2}
- \sqrt{\sum_{j\in \bbZ^*}  \| P_{n,j}(h) \|^2_2 -\sum_{j\in \bbZ^*\setminus S} \| P_{n,j}(h)\|^2_2}
\notag\\
&&\leq \| h\|_2-\sqrt{\sum_{j\in S}  \| P_{n,j}(h)\|^2_2}  \,\leq \, \frac{\epsilon^2}{32\|f\|_2}.
\end{eqnarray}
On the other hand, 
\begin{equation}\label{eq2-estimate}
\sqrt{\sum_{j\in \bbZ^*} \| P_{n,j}(h)\|^2_2}
+ \sqrt{\sum_{j\in \bbZ^*}  \| P_{n,j}(h) \|^2_2 -\sum_{j\in \bbZ^*\setminus S} \| P_{n,j}(h)\|^2_2} \,\leq 2\| h\|_2\,\leq \, 2\|f\|_2.
\end{equation}
Multiplying \eqref{eq1-estimate} and \eqref{eq2-estimate} together, finishes the proof, as we get 
$\sum_{j\in \bbZ^*\setminus S} \| P_{n,j}(h)\|^2_2 \,\leq \, \frac{\epsilon^2}{16}.$
\end{proof}

As a direct corollary of Theorem \ref{thm:convergence}, we now have the desired result that the graph Fourier transform converges to the graphon Fourier transform, when applied to a converging sequence of graph signals.
As mentioned in Remark~\ref{remark:GFT-WFT},  in the following corollary, graph signals are considered as vectors in $\bbC^n$, and the graph Fourier transform is as defined in \eqref{GFT}. The limiting graphon transform is the projection as defined in Definition \ref{def:NC-Fourier}.
\begin{corollary}\label{cor:GraphConvergence}
Fix a graphon $w:X\times X\rightarrow [0,1]$ and a graphon signal $f\in L^2(X)$, and  consider the sequence $\{(G_n , f_n)\}$
of graph signals converging to the graphon signal $(w, f)$.
Suppose that the graphs $G_n$ and the graph signals $f_n$ are labeled so that  $(w_{G_n,X},f_n^X)$ converges in norm to $(w,f)$.
Then graph Fourier transforms $\widehat{f}_n$ converge to the graphon Fourier transform $\widehat{f}$ in the following sense:
$$
\text{For each nonzero eigenvalue } \mu_j \text{ of } T_w,\ \sum_{i\in I_{\mu_j}} \widehat{f_n}(\phi_i^n) ({\phi^n_i})^X \rightarrow \widehat{f}(\mu_j)\mbox{ as } n\rightarrow \infty,
$$
where for each $n$, the adjacency matrix of $G_n$ has eigenvalues $\{\lambda^n_i\}$, ordered as in (\ref{eq:ordering}), with corresponding eigenvectors $\phi^n_i$.
Moreover, if $P_0^w(f)={\mathbf 0}$, then
$$
\sum_{j\in \bbZ^*}\|\sum_{i\in I_{\mu_j}} \widehat{f_n}(\phi_i^n) ({\phi^n_i})^X - \widehat{f}(\mu_j)\|_2^2\rightarrow 0\mbox{ as } n\rightarrow \infty.
$$
\end{corollary}
\begin{proof}
By Lemma~\ref{lem:graph-graphon-spec}, the sequence  $\{\lambda^n_i\}$  gives nonzero eigenvalues of the graphon  $w_{G_n,X}$ listed as in (\ref{eq:ordering}),
with corresponding eigenvectors $(\phi^n_i)^X$. \emph{Wlog} assume $G_n$ has $n$ vertices.
So, if $\mu_j\neq 0$, we have
\begin{eqnarray*}
P^{w_{G_n,X}}_{I_{\mu_j}}(f_n^X)=\sum_{i\in I_{\mu_j}}\langle f_n^X, (\phi^n_i)^X\rangle_{L^2(X)}(\phi^n_i)^X
=\sum_{i\in I_{\mu_j}}\langle f_n, \phi^n_i\rangle_{{\bbC}^{n}}(\phi^n_i)^X=\sum_{i\in I_{\mu_j}}\widehat{f_n}(\phi^n_i)(\phi^n_i)^X.
\end{eqnarray*}
Now, applying Theorem~\ref{thm:convergence} to the converging graph signal sequence $(w_{G_n,X},f_n^X)\to (w,f)$ finishes the proof. 
%
%
\end{proof}

For non-derogatory graphons, this corollary strengthens the previously known convergence result from \cite{RuizChamonRibeiro21}. Namely, suppose $\mu_j$ has multiplicity 1, so $I_{\mu_j}=\{\lambda_j\}$. Then $P^w_{\mu_j}(f) =\langle f, \phi_j\rangle\phi_j$, where $\phi_j$ is the $\lambda_j$-eigenvector of $w$. The corollary then states that
$$
\widehat{f_n}(\phi_j^n) {(\phi^n_j)^X} \rightarrow \langle f, \phi_j\rangle\ \phi_j \mbox{ as } n\rightarrow \infty.
$$
Since the functions $(\phi^n_j)^X$ and  $\phi_j$ are elements of $L^2(X)$ with unit norm, this implies that 
$$
\widehat{f_n}(\phi_j^n)\rightarrow \langle f, \phi_j\rangle \mbox{ as } n\rightarrow \infty.
$$
In addition, if $f$ is $c$-bandlimited for some $c>0$, then $P_0^w(f)={\mathbf 0}$, and we get
$$
\sum_{j\in\bbZ^*}|\widehat{f_n}(\phi_j^n)- \langle f, \phi_j\rangle |_2^2\rightarrow 0\mbox{ as } n\rightarrow \infty.
$$
Note that the scaling factor in the theorem of~\cite{RuizChamonRibeiro21} does not appear here. This is due to the fact that we have incorporated a
scaling factor of $\sqrt{n}$ in \eqref{eqn:fX}.

\subsection{Interpretation of Theorem \ref{thm:convergence} and its applications}
The graphon Fourier transform as introduced in Definition~\ref{def:NC-Fourier} is a vector-valued transform, which provides a decomposition for any given signal into projections of the signal onto each eigenspace of $T_w$. 
This definition differs from the previously known approach, where graphon Fourier transform was modeled after classical (Abelian) harmonic analysis, and the Fourier coefficients were simply defined as real/complex numbers.
The necessity for Definition~\ref{def:NC-Fourier} becomes apparent when one deals with graphons which possess eigenvalues of higher multiplicities.
In such cases, convergence only occurs at the level of eigenspaces. 

Suppose a graphon has an eigenvalue $\lambda$ with multiplicity $k$. Due to random fluctuations, samples from the graphon will likely have $k$ distinct eigenvalues close to $\lambda$. Our result indicates that the space spanned by the eigenvectors of those $k$ eigenvalues will be increasingly similar to the eigenspace of the graphon corresponding to $\lambda$, {as the size of the sampled graph increases}. However, there is no guarantee that the individual eigenvectors of the samples converge. 
We therefore argue that if several eigenvalues of the graph sequence converge to a single (repeated) eigenvalue of the limit graphon, then the corresponding eigenvectors  should be considered in their totality, and not individually. 

A special case  occurs when $T_w$ is of finite rank. Here, the set $I_w$ of nonzero (repeated) eigenvalues is finite. We first note  that, in this case, the second convergence result (\ref{eq:converg2}) follows directly from (\ref{eq:converg}), and thus the condition $P^w_0(f)=\mathbf{0}$ is not necessary.  Second, we observe that, for large $n$, a sample graph $G\sim \mathcal{G}(n,w)$ drawn from a finite rank graphon $w$ will likely have more non-zero eigenvalues than $w$. Namely, since edges are chosen independently at random, the rank of the adjacency matrix of a $w$-random graph, and thus its number of non-zero eigenvalues,   will likely grow to infinity as the size increases. 

As a simple example, let $\{G_n\}$ be a sequence of $w$-random graphs of increasing size, sampled from a constant graphon $w\equiv p$. We know from Theorem \ref{thm:szegedy-spectra} that the sequence $\{\lambda_1^n\}$, consisting of the largest eigenvalue of the adjacency matrix of each graph $G_n$, will converge to  $\lambda_1=p$, while all smaller eigenvalues will converge to zero.  By Theorem \ref{thm:convergence}, the eigenvectors of $G_n$ corresponding to the eigenvalues $\lambda^n_i$ with $i>1$ will converge to the kernel of $T_w$, and thus will not play a role in the spectral decomposition of $T_w$. 
A similar situation occurs for any finite rank graphon. That is, for any index $j$ outside $I_w$, the sequence of eigenvalues 
$\{\lambda_j^n\}_n$ converges to 0, and the associated sequence of eigenvectors converges to the kernel of $T_w$.
Our results suggest that such eigenvectors should be considered as sampling noise. Thus, an efficient analysis of the graph Fourier transform should only focus on eigenvalues with indices in $I_w$.

%
We suggest an approach for a unified Fourier analysis applicable to all graphs sampled from a given graphon $w:X\times X\to [0,1]$. 
Namely, we can propose as a graph Fourier transform, the projection onto eigenspaces of $T_w$. 
Our results show that, for large graphs, this Fourier transform will be similar to the GFT derived from the spectral decomposition of the adjacency matrix of the graph itself. 
\blue{This viewpoint is similar to the transferability results in \cite{neural1,neural2} for graph neural networks.}
\begin{example}[Watts-Strogatz model]\label{exp:Watts-Strogatz}
Consider the graphon $w:[0,1]^2\rightarrow [0,1]$ defined as follows. For all $x,y\in [0,1]$, let
$$
w(x,y)=\left\{ \begin{array}{ll}
1-p &\mbox{if }|x-y|\leq d\mbox{ or }|x-y|\geq 1-d,\\
p & \mbox{otherwise},
\end{array}\right.
$$
where $p,d\in (0,\frac{1}{2})$ are parameters of the model. 
The graphon $w$ is a Cayley graphon on the 1-dimensional torus (see Example \ref{exp:torus} for details). 
Random graphs drawn from $w$ have a natural circular layout: each vertex can be identified with a point $e^{2\pi i x}$ on the unit circle. Then each vertex is connected with probability $1-p$ to vertices that are close (in angular distance), and with probability $p$ to any other vertex. When $p$ is small, this graphon corresponds closely to the Watts-Strogatz model first proposed in \cite{WattsStrogatz98}, which is widely used to model so-called ``small-world'' networks. 
\begin{figure}[ht]
\centerline{\includegraphics[scale = 0.25]{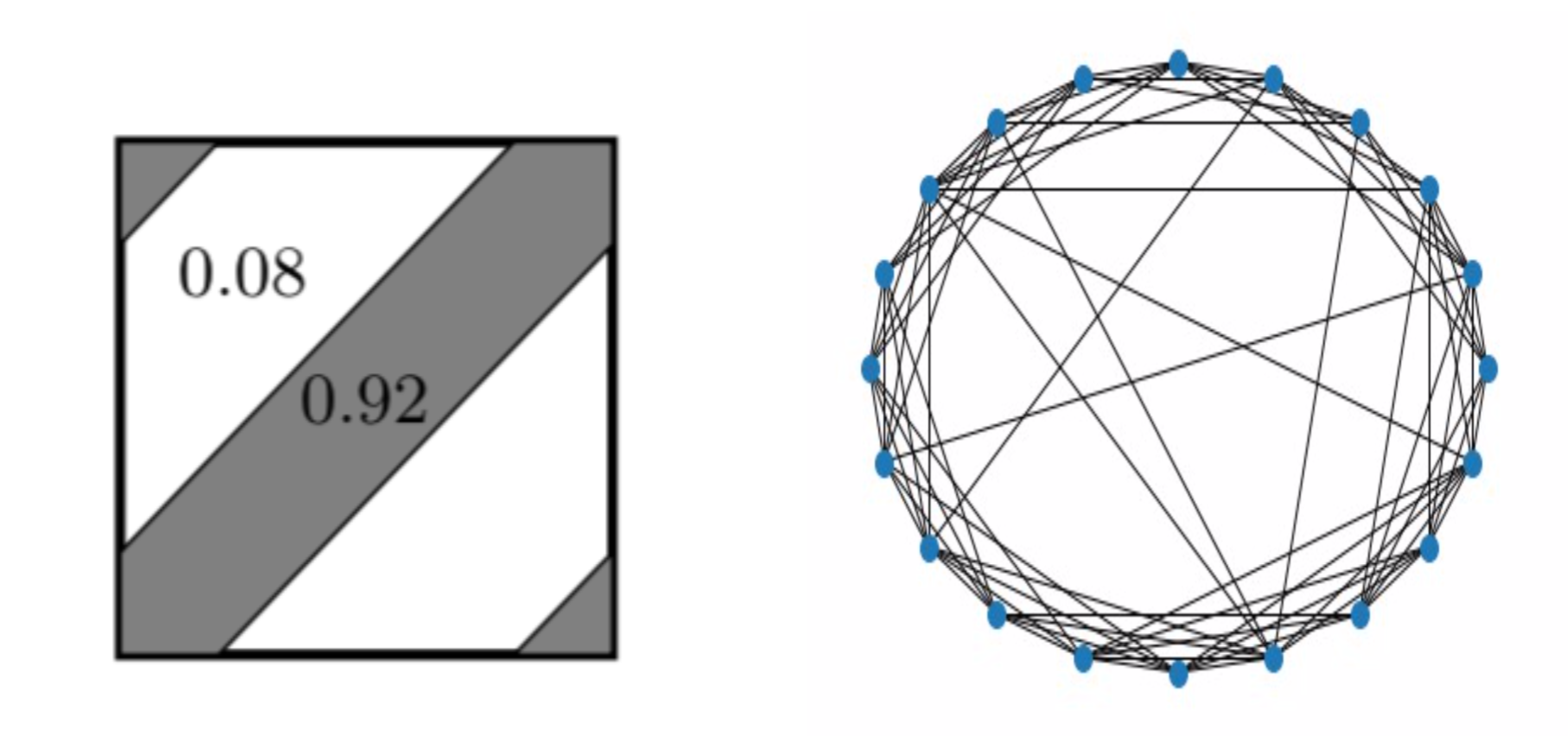}}
\caption{\blue{Cayley graphon on the 1-dim torus with parameters $d=0.2, p=0.08$ and a graph sampled from it.}}
\label{fg:projections}
\end{figure}

A straightforward calculation shows that the eigenvalues of $T_w$ are 
$$\left\{\frac{(1-2p)\sin(2\pi kd)}{\pi k}: k\in \bbZ^*\right\}\cup \{p+2d-4pd\}.$$
Taking $d=p=0.1$ and using notation as in \ref{notation:notation-thm1}, the first three eigenvalues are 
$$\lambda_1=0.6p+0.2, \lambda_2=\lambda_3=\left(\frac{1-2p}{\pi}\right)\sin(0.2\pi).$$ 
Then $\mu_2=\lambda_2=\lambda_3$, $I_{\mu_2}=\{ 2,3\}$, and the eigenspace corresponding to $\mu_2$ has dimension 2. 

Let $\{ G_n\}$ be a sequence of $w$-random graphs $G_n\sim {\mathcal{G}}(n,w)$. Our convergence result tells us that for large  $n$, the adjacency matrix of $G_n$, interpreted as a graphon, will have second and third largest positive eigenvalues $\lambda_2^n$ and $\lambda_3^n$ close to $\mu_2$. However, due to stochastic variation it is unlikely that $\lambda_2^n=\lambda_3^n$. Corollary \ref{cor:GraphConvergence} tells us that the space spanned by the $\lambda_2^n$- and $\lambda_3^n$-eigenvector converges to the eigenspace corresponding to $\mu_2$ (in the sense of the convergence of the associated orthogonal projections). It does not follow, {and is likely not true}, that the sequences  $\{\widehat{f}(\lambda^n_2)\}$ and  $\{\widehat{f}(\lambda^n_3)\}$ each converge.
We can then conclude that the graph Fourier coefficients $\widehat{f}(\lambda^n_2)$ and $\widehat{f}(\lambda^n_3)$ have little significance individually, but should be considered jointly. 
\end{example}

\blue{
\begin{example}\label{ex:S3_matrix}
Consider the graphon $w_M\in\mathcal{W}_0$ represented by the model matrix $M$, where 
$$M=\begin{bmatrix} 0.6 & 0.3 & 0.1 & 0 & 0 & 0 \\
                             0.3 & 0.6 & 0 & 0 & 0.1 & 0 \\
                             0.1 & 0 & 0.6 & 0.3 & 0 & 0 \\
                             0 & 0 & 0.3 & 0.6 & 0 & 0.1\\
                             0 & 0.1 & 0 & 0 & 0.6 & 0.3 \\
                             0 & 0 & 0 & 0.1 & 0.3 & 0.6 \\
 \end{bmatrix},$$
and $w_M$ has constant value $M_{i,j}$ on the sets $I_i\times I_j$, for $1\leq i,j\leq 6$. This is a Cayley graphon on the group $\mathbb{S}_3$; see also Example \ref{exp:S3}.

 We sample from this graphon as follows. We generate a graph $G_{N}$ with vertex set $V=\bigcup_{i=1}^6 V_i$, where $|V_i|=N$ for $1\leq i\leq 6$. Edges are added independently, with edge probabilities given by the model matrix. That is, the probability that a vertex $x\in V_i$ and $y\in V_j$ form an edge equals $M_{i,j}$. It is straightforward to show that the sequence $G_N$ converges to $w_M$.
 
 The six eigenvalues of the model matrix are shown in the top row of the table below. Note that there are two pairs of eigenvalues with multiplicity 2.  We sample ten graphs $G_N$ according to the process described above, with $N=1000$.  The adjacency matrices of the sampled graphs all have more than 6 non-zero eigenvalues. The first six eigenvalues of the samples are very similar to the non-zero eigenvalues of the model matrix, as predicted by Theorem \ref{thm:szegedy-spectra}. The seventh eigenvalue of the sample matrices demonstrates that the eigenvalues beyond the sixth eigenvalue converge to zero.
 
  \begin{table}[h]
     \centering
 \begin{tabular}{|l|c|c|c|c|c|c|c|}\hline
                        & $\lambda_1$ & $\lambda_2$   &  $\lambda_3$ & $ \lambda_4$    & $\lambda_5$    & $\lambda_6$  & $\lambda_7$ \\ \hline
 Model:            & 1.0000 & 0.8646 & 0.8646 & 0.3354 & 0.3354 & 0.2000 &  0 \\
 \hline
 Sample 1:       & 1.0005 & 0.8653 & 0.8651 & 0.3371 & 0.3357 & 0.2020 & 0.0457 \\
 Sample 2:       & 0.9999 & 0.8648 & 0.8642 & 0.3365 & 0.3352 & 0.2015 & 0.0457 \\
 Sample 3:       & 1.0001 & 0.8648 & 0.8643 & 0.3373 & 0.3364 & 0.2023 & 0.0457 \\
 Sample 4:       & 1.0001 & 0.8644 & 0.8644 & 0.3374 & 0.3367 & 0.2025 & 0.0457 \\
 Sample 5:       & 1.0000 & 0.8656 & 0.8642 & 0.3369 & 0.3364 & 0.2025 & 0.0457 \\
 Sample 6:       & 0.9998 & 0.8645 & 0.8639 & 0.3362 & 0.3357 & 0.2014 & 0.0457 \\
 Sample 7:       & 1.0004 & 0.8650 & 0.8646 & 0.3367 & 0.3358 & 0.2017 & 0.0458 \\
 Sample 8:       & 0.9998 & 0.8647 & 0.8638 & 0.3360 & 0.3354 & 0.2012 & 0.0458 \\
 Sample 9:       & 0.9997 & 0.8649 & 0.8635 & 0.3376 & 0.3370 & 0.2028 & 0.0457 \\
 Sample 10:     & 0.9998 &  0.8646 & 0.8643 & 0.3369 & 0.3367 & 0.2025 & 0.0457 \\ \hline
 \end{tabular}
      \caption{Eigenvalues of the model matrix and the adjacency matrices of the samples.}
     \label{tab:eigenvalues}
      \end{table}

Using a signal $f$ that is 1 on $V_1$ and zero elsewhere, we compute the graph Fourier coefficients. The results are given in Table \ref{tab:GFT}. We see that for coefficients 1 and 6, which correspond to eigenvalues with multiplicity 1, the values of all 10 samples are very similar. (Apart from the difference in sign, which is due to the fact that eigenvectors are unique up to sign.) However, this is not the case for the coefficients corresponding to eigenvalues with higher multiplicities.
To illustrate this, we focus on coefficients
2 and 3, computed from the projections of $f$ onto eigenvectors 2 and 3 ($\phi_2, \phi_3$) of each of the sampled graphs.

 \begin{table}[h]
  \begin{tabular}{|l|c|c|c|c|c|c|c|}\hline
&                         $\langle f,\phi_1\rangle $ &  $\langle f,\phi_2\rangle $ &   $\langle f,\phi_3\rangle $ &   $\langle f,\phi_4\rangle $ &  $\langle f,\phi_5\rangle $ &  $\langle f,\phi_6\rangle $  &  $\langle f,\phi_7\rangle $ \\ \hline
    Sample 1 &       -12.8538 & -13.8164 & -11.9317& -5.1438& -17.4057& -12.9621& 0.0178 \\
Sample 2 & -12.8774 &  -15.3514 & 9.8731 & -1.4261 & -18.1699 & -12.8366 & -0.0057 \\
Sample 3 & -12.9054 & -18.1773 & -1.6169 & -16.9128 & -6.8229 & 12.7908 & 0.0168 \\
Sample 4 & 12.9197 & -11.6449 & -14.0502 & -9.4462 & 15.5340 & 12.8565 & -0.0224\\
Sample 5 &-12.8981 & -12.9253 &12.8846 & -12.4981 & 13.2536 & 12.8339 & -0.0071\\
Sample 6 & -12.9283 & -16.6597 & 7.3625 & 11.5344 & 14.1140 & -12.8307 & 0.0622 \\
Sample 7 & 12.8443 & 5.5619 & -17.4151 & 9.6619 & -15.4254 & -12.8584 & 0.0281 \\
Sample 8 & 12.8231 & 11.0069 & -14.6563 & 18.2037 & 0.7167 & -12.7901 & 0.0715 \\ 
Sample 9 & -12.9264 & -16.9912 & 6.6391 & -3.8731 & -17.8276 & 12.7807 & -0.0148\\
Sample 10 & -12.9554 & 10.4365 & 14.9858 & 17.3406 & -5.6170 & -12.7360 & 0.0066 \\ \hline
\end{tabular}
\caption{Fourier coefficients of the sampled graphs.}
\label{tab:GFT}
\end{table}

 In the sample graphs, the eigenvectors $\phi_2, \phi_3$ do not belong to a single eigenspace. As shown in Table~\ref{tab:GFT}, 
 the modulus of the individual graph Fourier coefficients corresponding to $\phi_2$ (respectively, $\phi_3$) do not converge.
 Figure \ref{fg:projections} shows that the graph Fourier coefficients $\widehat{f}(\phi_2)=\langle f,\phi_2\rangle$ and $\widehat{f}(\phi_3)=\langle f,\phi_3\rangle$ vary greatly from sample to sample. The blue dots are the projections of the signal $f$ onto the eigenvectors $\phi_2$ and $\phi_3$ of the sampled graphs: the $x$-coordinate of each dot is the inner product $\langle f,\phi_2\rangle$, and the $y$-coordinate equals $\langle f,\phi_3\rangle$. The dots do not cluster together, since the individual values of $\langle f,\phi_2\rangle$ and $\langle f,\phi_3\rangle$ do not converge. However, the dots lie close to  a circle; this shows that the length of the projection onto the eigenspace spanned by $\phi_2$ and $\phi_3$ does converge. This is indeed an easy consequence of the fact that the vectors $\langle f,\phi_2\rangle\phi_2+\langle f,\phi_3\rangle\phi_3$ converge as the size of $G_N$ grows (by Theorem~\ref{thm:convergence}).

In Section \ref{sec:cayley}, we will show how to choose a basis to define a graph Fourier transform that is sampled from a  Cayley graphon.
The red diamond shows the projection of $f$ onto the vectors of this basis which correspond to eigenvalue $\mu_2$. As predicted by the theory, this projection falls on the same circle.

\begin{figure}[ht]
\centerline{\includegraphics[scale = 0.6]{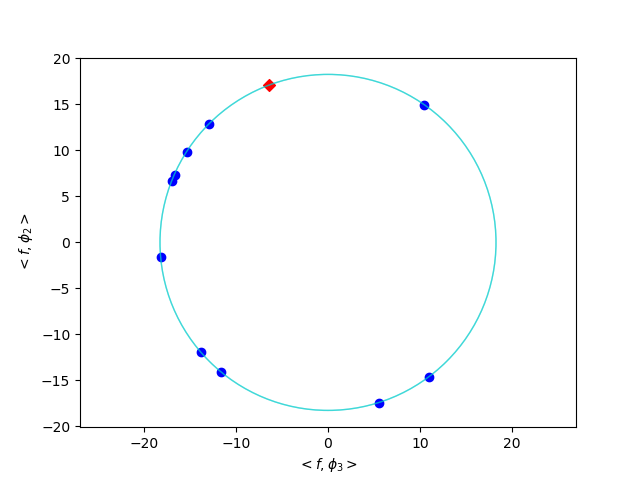}}
\caption{Graph Fourier coefficients 2 and 3 for the different samples.}
\label{fg:projections}
\end{figure}
\end{example}}

\subsection{Application: Filter Design}

In graph signal processing, the GFT guides the design of graph filters. Diffusion of a graph signal reflects the structure of the graph. Therefore, the graph shift operator $S$ is often taken to be the adjacency matrix $A$. A \emph{polynomial graph filter} $H$ on a graph with $n$ vertices is any polynomial in $A$ (see for example \cite{2018:Ortega:GSPOverview}):
$$
H=\sum_{k=0}^m h_k A^k.
$$
Let $h$ be the polynomial $h(x)=\sum_{k=0}^m h_k x^k$. It follows directly from the definition of GFT and the spectral decomposition of the adjacency matrix that, for each eigenvalue $\lambda_i$ of $A$ with associated eigenvector $\phi_i$:
\begin{equation}\label{eq:graphfilter}
\widehat{Hf}(\phi_i)=h(\lambda_i)\widehat{f}(\phi_i).
\end{equation}
As proposed in \cite{MorencyLeus21}, this approach can be extended to graphons as follows.  The shift operator of a graphon $w:X\times X\rightarrow [0,1]$ is the associated operator $T_w$, and a graphon filter is likewise defined as a polynomial in $T_w$:
$$
H=\sum_{k=0}^m h_kT_w^k.
$$
Using the spectral decomposition of $T_w$ and adopting the notation from \ref{notation:notation-thm1}, we have that, for each $f\in L^2(X)$,
$$
Hf=h_0f+ \sum_{k=1}^m h_k \sum_{j:\mu_j\neq 0} \mu_j^k P^w_{I_{\mu_j}}(f).  
$$
As before, let $h$ be the polynomial $h(x)=\sum_{k=0}^m h_kx^k$. Using our extended definition of the graphon Fourier transform as given in Definition \ref{def:NC-Fourier}, we then have that, for all $\mu_j\neq 0$, 
\begin{equation}\label{eq:graphonfilter}
\widehat{Hf}(\mu_j)=P^w_{I_{\mu_j}}(Hf)=h(\mu_j)P_{I_{\mu_j}}^wf=h(\mu_j) \widehat{f}(\mu_j).
\end{equation}
Our convergence results then immediately imply the convergence of the filter response as stated below.
\begin{corollary}
Let $\{(G_n,f_n)\}$ be a sequence of  graph signals converging to a graphon signal $(w,f)$, and assume that the graphs $G_n$ and the graph signals $f_n$ are labeled so that  $(w_{G_n,X},f_n^X)$ converges in norm to $(w,f)$.

\green{For each $n$, let $A_n$ be the adjacency matrix of $G_n$,  and let its eigenvalues be denoted as $\{\lambda^n_i\}$, labeled as in (\ref{eq:ordering}), with corresponding eigenvectors $\phi^n_i$. Given a polynomial $h(x)=\sum_{k=0}^m h_kx^k$, for each $n$ let $H_{n}=\sum_{k=0}^m h_k A_n^k$.}
Then for each nonzero eigenvalue $\mu_j$ of $T_w$,
$$
\sum_{i\in I_{\mu_j}} \widehat{H_{n}f_n}(\phi_i^n) ({\phi^n_i})^X \rightarrow h(\mu_j)\widehat{f}(\mu_j)\mbox{ as } n\rightarrow \infty.
$$

%
\end{corollary}
\begin{proof}
The first statement follows directly from Corollary~\ref{cor:GraphConvergence}, Equations \eqref{eq:graphfilter} and \eqref{eq:graphonfilter},  and the fact that 
$\lim_{n\to \infty}\lambda_i^n=\mu_j$ for each $i\in I_{\mu_j}$.
\end{proof}

This corollary gives strong evidence that, for large graphs sampled from a graphon $w$,  one should design graph filters with respect to the limiting graphon, rather than the graph itself. Also, when evaluating the effect of a filter on GFT, one should consider the Fourier coefficients of eigenvalues with indices in $I_{\mu_j}$ as a whole.

\section{Signal processing on Cayley graphons}\label{sec:cayley}
The instance-independent approach presented in this article is particularly favorable in the special case that the limit graphon is a Cayley  graphon.
In this case, we have the well-established and rich theory of group representations at our disposal, which we employ to obtain suitable graphon Fourier bases. Fourier analysis informed by representation theory of the Cayley graphon can lead to decomposition of signals into `meaningful' components; for an instance of this phenomenon, see \cite{permutahedron}.
Cayley graphons reflect the symmetries of the underlying group, and can be used to model real-life networks. For example, the Watts-Strogatz model from Example~\ref{exp:Watts-Strogatz} is a Cayley graphon. \blue{Also, graphons derived from the symmetric group $\mathbb{S}_k$ can be used to represent ranked data (see Example~\ref{exp:S3})}. In this section, we show how the representations of the underlying group naturally yield the spectral decomposition of the associated Cayley graphon, and can be used to define a universal GFT for samples from the graphon.

We fix the following notations throughout this section:  let $w$ be a Cayley graphon  on a  compact group $\bbG$ defined by a Cayley function $\gamma:\bbG\rightarrow [0,1]$  (see Section \ref{subsec:general-graphon} and Definition \ref{def:cayley-graphon} for precise descriptions). 

Applying Fourier analysis of non-Abelian groups  as discussed in Subsection \ref{subsec:FT-nonAbelian}, we obtain properties of the eigenvalues/eigenvectors of $T_w$, which we list in Theorem \ref{thm:eigenvector}. In the next lemma, we will see that the action of $T_w$ on a signal $f$ can be expressed in terms of a convolution operator, which is computationally preferred when dealing with representations.
%
\begin{lemma}\label{lem:Tconv}
Let $f\in L^2(\bbG)$. For almost every $x\in \bbG$, we have $T_w(f)(x)=(\widecheck{f}*\widecheck{\gamma})(x^{-1})$, where the ``check operation'' $f\mapsto \widecheck{f}$ on $L^1(\bbG)$ is defined as $\widecheck{f}(x)=f(x^{-1})$. Consequently, we have 
$$\widecheck{T_w(f)}(x)={(\widecheck{f}*\widecheck{\gamma})}(x).$$
\end{lemma}
\begin{proof}
For almost every $x\in \bbG$, we have $T_w(f)(x)=\int_{\bbG} w(x,y)f(y)\, dy=\int_{\bbG} \gamma(xy^{-1})f(y)\, dy$. So, applying the change of variable $y\mapsto y^{-1}$, we have
\begin{equation*}
T_w(f)(x)=\int_{\bbG} \widecheck{f}(y^{-1})\widecheck{\gamma}(yx^{-1})\, dy=\int_{\bbG} \widecheck{f}(y)\widecheck{\gamma}(y^{-1}x^{-1})\, dx=( \widecheck{f}*\widecheck{\gamma})(x^{-1}).
\end{equation*}
\end{proof}
As we will see in Theorem~\ref{thm:eigenvector}, the spectral analysis of matrices  $\pi(\gamma)$ play a central role in the spectral decomposition of $T_w$. 
\begin{lemma}\label{lem:pi(f)sa}
Let $\gamma$ be a Cayley function on a group $\bbG$, i.e.~$\gamma(x)=\gamma(x^{-1})$ for all $x\in \bbG$.
Then for every unitary representation $\pi:\bbG\to {\mathcal U}({\mathcal H}_\pi)$, the operator $\pi(\gamma)$ is self-adjoint. In particular, $\pi(\gamma)$ is diagonalizable, and its spectrum lies in ${\mathbb R}$. 
\end{lemma}
\begin{proof}
Recall that $\pi(\gamma)\in {\mathcal B}({\mathcal H}_\pi)$ is defined as $\int_{\bbG} \gamma(x)\pi(x)\, dx$, where the integration is with respect to the Haar measure of $\bbG$. This integral should be interpreted weakly, that is,
$$\left\langle\left(\int_{\bbG} \gamma(x)\pi(x) \, dx\right)\xi,\eta\right\rangle= \int_{\bbG} \gamma(x)\langle\pi(x)\xi,\eta\rangle \, dx,$$
for each $\xi,\eta$ in the Hilbert space of $\pi$. For an arbitrary pair  $\xi,\eta\in{\mathcal H}_\pi$, we have
\begin{eqnarray*}
\left\langle\left(\int_{\bbG} \gamma(x)\pi(x) \, dx\right)^*\xi,\eta\right\rangle&=&\overline{\left\langle\left(\int_{\bbG} \gamma(x)\pi(x) \, dx\right)\eta,\xi\right\rangle}
=\overline{\int_{\bbG} \gamma(x)\langle\pi(x)\eta,\xi\rangle \, dx}\\
&=&\int_{\bbG} \overline{\gamma(x)}\overline{\langle\pi(x)\eta,\xi}\rangle \, dx
=\int_{\bbG} \gamma(x)\langle\pi(x^{-1})\xi,\eta\rangle \, dx\\
&=&\int_{\bbG} \gamma(x^{-1})\langle\pi(x)\xi,\eta\rangle \, dx=\langle\pi(\gamma)\xi,\eta\rangle,
\end{eqnarray*}
where we used the change of variable $x\mapsto x^{-1}$, and the fact that $\gamma(x)=\gamma(x^{-1})$.
\end{proof}
\begin{theorem}\label{thm:eigenvector}
Let $w:\bbG\times \bbG\to [0,1]$ be the Cayley graphon defined by a Cayley function $\gamma:\bbG\to [0,1]$ on a compact group $\bbG$.
\begin{itemize}
\item[(i)] The set of eigenvalues of $T_w$ is given as $\bigcup_{\pi\in\widehat{\bbG}}\left\{\mbox{eigenvalues of }\ \pi(\gamma)\right\}$.
\item[(ii)] For every nonzero eigenvalue $\lambda$ of $T_w$, there are  finitely many $\pi\in \widehat{\bbG}$ such that $\lambda\in{\rm Spec}(\pi(\gamma))$. We denote this finite set by $\widehat{\bbG}_{\lambda,\gamma}$.
\item[(iii)] Let $0\neq \lambda$ be an eigenvalue of $T_w$. Then $\lambda$-eigenvectors $\phi\in L^2(\bbG)$ can be characterized as 
$$\phi(x)=\sum_{\pi\in \widehat{\bbG}_{\lambda,\gamma}} d_\pi \overline{{\rm Tr}[A_\pi\pi(x)^*]},$$
where $A_\pi$ is a matrix with the property that every one of its columns is either zero, or a $\lambda$-eigenvector for $\pi(\gamma)$.
(Note that at least one of the $A_\pi$'s must be nonzero.)
\item[(iv)] The multiplicity of every nonzero eigenvalue $\lambda$ of $T_w$ is given by 
$\sum_{\pi\in \widehat{\bbG}_{\lambda,\gamma}} d_\pi m_{\lambda,\pi},$
where  $m_{\lambda,\pi}$ is the multiplicity of the eigenvalue $\lambda$ for $\pi(\gamma).$
\end{itemize}
\end{theorem}
\begin{proof}
To prove (i), suppose $0\neq \phi\in L^2(\bbG)$ is a $\lambda$-eigenvector of $T_w$, i.e.~$T_w(\phi)=\lambda \phi$ in $L^2(\bbG)$. 
By Lemma~\ref{lem:Tconv}, and the fact that $\gamma$ is a Cayley function
(i.e.~$\widecheck{\gamma}=\gamma$), this identity can be written as  $\widecheck{\phi}*{\gamma}=\lambda\widecheck{\phi}$.
%
Consequently, for every $\pi\in\widehat{\bbG}$, $\pi(\widecheck{\phi}*\gamma)=\pi(\widecheck{\phi})\pi(\gamma)=\lambda\pi(\widecheck{\phi})$. So by injectivity of the Fourier transform, $\phi$ is a $\lambda$-eigenvector of $T_w$ precisely when  for every $\pi\in\widehat{\bbG}$, we have
\begin{equation}\label{eq1-pf-cayley}
    \pi(\widecheck{\phi})(\pi(\gamma)-\lambda I_{d_\pi})=0,
\end{equation}
where $I_{d_\pi}$ is the identity matrix of dimension $d_\pi$. 
Taking matrix-adjoint from both sides of Equation \eqref{eq1-pf-cayley}, this equation can be written as 
\begin{equation}\label{eq:conv-fouroer}
(\pi(\gamma)-\lambda I_{d_\pi})\pi(\overline{\phi})=0  \ \mbox{ for every } \pi\in\widehat{\bbG}.
\end{equation}
Thus, we have:
\begin{itemize}
\item[(a)] If $\lambda$ is not an eigenvalue of $\pi(\gamma)$, then $\pi(\gamma)-\lambda I_{d_\pi}$ is invertible. So, $\pi(\overline{\phi})=0$.
\item[(b)] If $\lambda$ is an eigenvalue of $\pi(\gamma)$, then every nonzero column of the matrix $\pi(\overline{\phi})$ must be a  $\lambda$-eigenvector of $\pi(\gamma)$.
\end{itemize}
As a result, if $\lambda$ is not an eigenvalue of $\pi(\gamma)$ for any $\pi\in\widehat{\bbG}$, then $\phi=0$, contradicting our assumption. So $\lambda$ is an eigenvalue of $T_w$ with associated eigenvector $\phi$ if and only if it is an eigenvalue of $\pi(\gamma)$ for some $\pi\in \widehat{\bbG}$. This finishes the proof of (i). 

To prove (ii), we apply the Parseval identity for $\gamma$ as follows:
\begin{eqnarray*}
\|\gamma\|_2^2=\sum_{\pi\in\widehat{\bbG}} d_\pi {\rm Tr}[\pi(\gamma)\pi(\gamma)^*]=\sum_{\pi\in\widehat{\bbG}} d_\pi \left(\sum_{\lambda\in{\rm Spec}(\pi(\gamma))}\lambda^2\right).
\end{eqnarray*}
Since the above sum is finite, for every given $\lambda\neq 0$, there are only finitely many $\pi$ with $\lambda \in{\rm Spec}(\pi(\gamma))$.

To prove (iii), assume $\lambda$ is a nonzero eigenvalue of $T_w$,  and recall that 
$$\widehat{\bbG}_{\lambda,\gamma}=\left\{\pi\in \widehat{\bbG}:\ \lambda\in {\rm Spec}(\pi(\gamma))\right\}.$$
From (a) and (b), $0\neq \phi\in  L^2(\bbG)$ is a $\lambda$-eigenvector of $T_w$ if and only if 
\begin{itemize}
\item[(a$'$)] $\pi(\overline{\phi})=0$ for all $\pi\in \widehat{\bbG}\setminus\widehat{\bbG}_{\lambda,\gamma}$.
\item[(b$'$)] If $\pi\in\widehat{\bbG}_{\lambda,\gamma}$,  every nonzero column of the matrix $\pi(\overline{\phi})$ must be a  $\lambda$-eigenvector of $\pi(\gamma)$.
\end{itemize}
Using the inverse group Fourier transform (Equation~(\ref{inverse-F-noncommutative})), we get $\overline{\phi}(x)=\sum_{\pi\in \widehat{\bbG}_{\lambda,\phi}} d_\pi {\rm Tr}[\pi(\overline{\phi})\pi(x)^*]$.
Since this is a finite sum, there are no convergence issues to be considered here. Letting $A_\pi=\pi(\overline{\phi})$ finishes the proof of (iii).

To prove (iv), fix a nonzero eigenvalue $\lambda$ and a $\lambda$-eigenvector $\phi$ of $T_w$. By part (iii) of this theorem, 
$${\phi}(x)=\sum_{\pi\in \widehat{\bbG}_{\lambda,\gamma}} d_\pi \overline{{\rm Tr}[A_\pi\pi(x)^*]},$$
where every nonzero column of $A_\pi$ is a $\lambda$-eigenvector for $\pi(\gamma)$. 
For every $\pi\in\widehat{\bbG}_{\lambda,\gamma}$, let ${\mathcal E}_{\lambda, \pi(\gamma)}$ denote a fixed basis for the $\lambda$-eigenspace of $\pi(\gamma)$. Recall that $m_{\lambda,\pi}=|{\mathcal E}_{\lambda, \pi(\gamma)}|$.
It then follows immediately, from the above expression, that $\phi$ can be written as a linear combination of functions of the form $x\mapsto \overline{{\rm Tr}[A^{\pi}_{X, i}\ \pi(x)^*]}$, where $A^{\pi}_{X,i}$ denotes the matrix of size $d_\pi$ whose $i$'th column is $X\in {\mathcal E}_{\lambda,\pi(\gamma)}$, and its every other column is zero.
Applying a simple counting argument, we obtain the upper bound $\sum_{\pi\in \widehat{\bbG}_{\lambda,\gamma}} d_\pi m_{\lambda,\pi}$  for the multiplicity of the eigenvalue $\lambda$ of $T_w$.

To finish the proof, we will obtain the same number of independent $\lambda$-eigenvectors for $T_w$. 
From the definition of coefficient functions (Subsection \ref{subsec:FT-nonAbelian}), we observe that the $(i,j)$th entry of  $\pi(x)^*$ equals $\overline{\pi_{i,j}(x)}$. 
For $Z\in {\mathcal E}_{\lambda, \pi(\gamma)}$ represented as $Z=[z_j]_{j=1}^{d_\pi}$,
we have 
\begin{equation*}
    \overline{{\rm Tr}[A^{\pi}_{Z,i}\pi(x)^*]}=\sum_{j=1}^{d_\pi} \overline{z_j}{\pi_{i,j}(x)} \in {\rm Span}\{\pi_{i,j}: \ j=1,\ldots, d_\pi\}.
\end{equation*}
The above equation, together with Schur's orthogonality relations (Proposition~\ref{prop:Schur-cpt-group}), implies that 
\begin{itemize}
    \item[(i)] If $i\neq j$, then the functions $\overline{{\rm Tr}[A^{\pi}_{Z,i}\pi(x)^*]}$ and $\overline{{\rm Tr}[A^{\pi}_{Z,j}\pi(x)^*]}$ are orthogonal nonzero functions in $L^2(\bbG)$. 
    \item[(ii)] If $\pi,\sigma\in \widehat{\bbG}_{\lambda,\gamma}$ are distinct (inequivalent) representations, then  for every $X\in {\mathcal E}_{\lambda, \pi(\gamma)}$ and $Y\in {\mathcal E}_{\lambda, \sigma(\gamma)}$, and every $1\leq i\leq d_{\pi}$ and $1\leq j\leq d_{\sigma}$, we have that $\overline{{\rm Tr}[A^{\pi}_{X,i}\pi(x)^*]}$ and $\overline{{\rm Tr}[A^{\sigma}_{Y,j}\sigma(x)^*]}$ are orthogonal nonzero functions in $L^2(\bbG)$. 
    \item[(iii)] If $Y,Z\in {\mathcal E}_{\lambda, \pi(\gamma)}$ are distinct, then  $\overline{{\rm Tr}[A^{\pi}_{Y,i}\pi(x)^*]}$ and
     $\overline{{\rm Tr}[A^{\pi}_{Z,i}\pi(x)^*]}$ are orthogonal nonzero functions in $L^2(\bbG)$. 
\end{itemize} 
Parts (i) and (ii) follow directly from the statement of Schur's orthogonality relations. To prove (iii), take  distinct (orthogonal) elements $Y=[y_k]_{k=1}^{d_\pi}$ and
$Z=[z_j]_{j=1}^{d_\pi}$ of ${\mathcal E}_{\lambda, \pi(\gamma)}$. Using orthogonality relations between $\pi_{i,j}$ and $\pi_{i,k}$, we get:
\begin{eqnarray*}
&&\langle\overline{{\rm Tr}[A^{\pi}_{Z,i}\pi(\cdot)^*]}, \overline{{\rm Tr}[A^{\pi}_{Y,i}\pi(\cdot)^*]}\rangle_{L^2(\bbG)}
=
\langle\sum_{j=1}^{d_\pi} \overline{z_j}{\pi_{i,j}}, \sum_{k=1}^{d_\pi} \overline{y_k}{\pi_{i,k}}\rangle_{L^2(\bbG)}\\
&=&
\ \ \sum_{j=1}^{d_\pi}\sum_{k=1}^{d_\pi} \overline{z_j}y_k\langle{\pi_{i,j}}, {\pi_{i,k}}\rangle_{L^2(\bbG)}
=
\frac{1}{d_\pi}\sum_{j=1}^{d_\pi}\overline{z_j}y_j=0.
\end{eqnarray*}
Thus, the set of functions $\left\{\overline{{\rm Tr}[A^{\pi}_{X,i}\pi(x)^*]}: \pi\in\widehat{\bbG}_{\lambda,\gamma},\ X\in {\mathcal E}_{\lambda, \sigma(\gamma)},\ 1\leq i\leq d_\pi\right\}$ forms a basis for the $\lambda$-eigenbasis of $T_w$; this finishes the proof.
\end{proof}

Theorem \ref{thm:eigenvector} reduces the problem of finding a spectral decomposition for $T_w$ to finding spectral decompositions of $\pi(\gamma)$ for each $\pi\in\widehat{\bbG}$. This application of representation theory leads to significant simplification of the problem. Indeed, $T_w$ is an operator on the infinite-dimensional space $L^2(X)$, and obtaining a spectral decomposition for $T_w$ is a nontrivial task; whereas each $\pi(\gamma)$ is a finite-dimensional matrix.  

A special case arises when  $\bbG$ is Abelian. In this case, every representation $\chi\in\widehat{\bbG}$ is 1-dimensional. The following corollary uses Theorem \ref{thm:eigenvector} to describe eigenvalues and eigenvectors of $T_w$, when $w$ is the Cayley graphon of a compact Abelian group. In the statement below, $\mathcal F$ is used to refer to the group Fourier transform as defined in \eqref{F-noncommutative}.
\begin{corollary}\label{cor:Abelian}
Let $w:\bbG\times \bbG\to [0,1]$ be the Cayley graphon of a  compact Abelian group $\bbG$ defined by a Cayley function $\gamma:\bbG\to [0,1]$. 
\begin{itemize}
\item[(i)] For every $\lambda\in \bbR$, define ${\mathcal U}_\lambda:=\{\chi\in \widehat{\bbG}:\ ({\mathcal F}{\gamma})(\chi)=\lambda\}$. The set of eigenvalues of $T_w$ can be described as $\{  ({\mathcal F}{\gamma})(\chi) :\ \chi\in \widehat{\bbG}\}=\{\lambda\in \bbR:\ {\mathcal U}_\lambda\neq \emptyset\}$.
\item[(ii)] Any nonzero $\phi\in L^2(\bbG)$ such that  ${\mathcal F}{\overline{\phi}}$ is supported on ${\mathcal U}_\lambda$ is a $\lambda$-eigenvector of  $T_w$.
\end{itemize}
\end{corollary}
The above corollary follows directly from Theorem~\ref{thm:eigenvector}. We demonstrate a more direct proof for the Abelian case in the following example.
\begin{example}[Graphons on the 1-dimensional torus]
\label{exp:torus}
Consider the Abelian compact group $\bbT=\{e^{2\pi i x}: x\in[0,1)\}$, with multiplication as the group product. Let $\gamma:\bbT\to [0,1]$ be a
Cayley function, i.e.~
$\gamma(x)=\gamma(x^{-1})$ for all $x\in \bbT$.
Using the identification of $\bbT$ and $[0,1)$, the Lebesgue measure on $[0,1)$ is transferred to the Haar measure on $\bbT$. Let $w:\bbT\times \bbT\rightarrow [0,1]$ be the Cayley graphon defined by $\gamma$. The integral operator associated with $w$ is defined as 
$$T_w:L^2(\bbT)\to L^2(\bbT), \ (T_wf)(x)=\int_{\bbT}\gamma(xy^{-1}) f(y)\, dy=(f*\gamma)(x),$$
where the last equality holds as $\bbT$ is Abelian. To find eigenvalues/eigenvectors of $T_w$, we use classical Fourier analysis on $\bbT$, noting that $\widehat{\bbT}\simeq\bbZ$. In this example, we write $\widehat{f}(n)$ to denote the $n$'th Fourier coefficient of $f$.
Suppose $f\neq 0$ is a $\lambda$-eigenvector of $T_w$. Then, we have the following equivalent relations:
\begin{eqnarray*}
T_wf=\lambda f \mbox{ in } L^2(\bbT) &\Leftrightarrow & f*\gamma=\lambda f \mbox{ in } L^2(\bbT)\\
&\Leftrightarrow &  \mbox{ for every } n\in \bbZ,\ \widehat{f}(n)\widehat{\gamma}(n)=\lambda \widehat{f}(n) \\
&\Leftrightarrow & \mbox{ for every } \ n\in\bbZ,  \widehat{f}(n)=0 \mbox{ whenever } \widehat{\gamma}(n)\neq \lambda.
\end{eqnarray*}
Let ${\mathcal U}_\lambda=\{n\in\bbZ: \ \widehat{\gamma}(n)=\lambda\}$. If ${\mathcal U}_\lambda=\emptyset$, then we must have $\widehat{f}\equiv 0$, and consequently $f=0$; this is a contradiction with the choice of $f$ as a $\lambda$-eigenvector. On the other hand, if ${\mathcal U}_\lambda\neq \emptyset$, then any nonzero function $f$ whose Fourier series is supported on ${\mathcal U}_\lambda$ is a $\lambda$-eigenvector for $T_w$. 

We note that in this particular example, we can replace ``${\mathcal F}{\overline{f}}$ is supported on ${\mathcal U}_\lambda$'' with the phrase 
``${\mathcal F}{{f}}$ is supported on ${\mathcal U}_\lambda$'' in the statement of Corollary \ref{cor:Abelian} (ii). This is due to the fact that 
(i) $\widehat{\overline{f}}(n)=\overline{\widehat{f}(-n)}$ for every $n\in\bbZ$, and (ii) ${\mathcal U}_\lambda$ is closed under negation as $\gamma$ is real-valued.
\end{example}

Next, we show how to obtain an eigenbasis for the integral operator of a Cayley graphon, using its harmonic analysis.
Harmonic analysis of non-Abelian compact groups is mainly focused on the study of the group representations and their associated function spaces. 
An important (not irreducible) unitary representation of a group $\bbG$ is the \emph{left regular representation}, defined as $L:\bbG\to \U(L^2(\bbG))$, $(L(g) f)(h)=f(g^{-1}h)$, for $f\in L^2(\bbG)$ and $g,h\in \bbG$. 
The integral operator of a Cayley graphon can be expressed in terms of the left regular representation of the underlying group. 

\begin{remark}\label{remark:leftregular}
Let $w$ be a graphon on a group $\bbG$ defined by a Cayley function $\gamma$. There is a direct relation between the integral operator $T_w$ and the left regular representation $L$. Namely,
for $f\in  L^2(\bbG)$ and almost every $x\in \bbG$,  we have
\begin{eqnarray*}
T_w(f)(x)=\int_{\bbG} w(x,y)f(y)\, dy 
=\int_{\bbG} \gamma(xy^{-1})f(y)\, dy
&=&\int_{\bbG} \gamma(y)f(y^{-1}x) \, dy\\
&=&\int_{\bbG} \gamma(y)(L(y)f)(x) \, dy
\end{eqnarray*}
So, $T_w(f)=L(\gamma)f$. 
\end{remark}
To develop signal processing on Cayley graphons, we use the Peter-Weyl basis of $L^2(\bbG)$. 
The Peter-Weyl theorem (\cite[Theorem 5.12]{1995:Folland:HarmonicAnalysis}) asserts that the left regular representation of $\bbG$ is unitarily equivalent to $\bigoplus_{\pi\in \widehat {\bbG}} d_\pi \pi$, where $d_\pi$ denotes the dimension of $\pi$. That is, every irreducible representation of $\bbG$ appears in the decomposition of $L$ with multiplicity equal to the dimension of the representation. 
The orthogonal decomposition presented in the Peter-Weyl theorem and the precise orthogonality relations amongst the irreducible pieces play a central role in in the proof of the following proposition.
\begin{proposition}[Eigenbasis for Cayley graphons]\label{prop:basis}
Let $\bbG$ be a second countable compact group, and consider the Cayley graphon $w:\bbG\times \bbG\to[0,1]$ obtained from the Cayley function $\gamma:\bbG\to [0,1]$. For each $\pi$, let ${\mathcal E}_{\overline{\pi(\gamma)}}$ denote a fixed eigenbasis for $\overline{\pi(\gamma)}$, where the matrix $\overline{\pi(\gamma)}$ is obtained from $\pi(\gamma)$ by taking complex conjugation entry-wise. 
Then 
$$\bigcup_{\pi\in\widehat{\bbG}}\bigcup_{i:1,\ldots,d_\pi}\left\{\sum_{j=1}^{d_\pi}z_j\pi_{i,j}: \ 
\left[
\begin{array}{c}
z_1 \\
\vdots \\
z_{d_\pi}   
\end{array}
\right]
\in {\mathcal E}_{\overline{\pi(\gamma)}} \right\}$$
is an (orthogonal) eigenbasis for $T_w$. 
\end{proposition}
\begin{proof}
Let $\pi\in\widehat{\bbG}$. First, note that for the coefficient function $\pi_{i,j}\in L^2(\bbG)$, we have
\begin{eqnarray}
T_w(\pi_{i,j})(y)&=&\int_{\bbG} w(x,y)\pi_{i,j}(x)\, dx
=\int_{\bbG} \gamma(xy^{-1})\langle \pi(x)e_i,e_j\rangle\, dx\nonumber\\
&=&\int_{\bbG} \gamma(x)\langle \pi(xy)e_i,e_j\rangle\, dx
=\langle \pi(\gamma)(\pi(y)e_i),e_j\rangle=\langle \pi(y)e_i,\pi(\gamma)e_j\rangle, \label{eq:1-last-prop}
\end{eqnarray}
where in the last equality we used the fact that $\pi(\gamma)$ is self-adjoint. 
Now suppose $\pi(\gamma)=[\alpha_{i,j}]$. So  for every $i$ we have the linear expansion $\pi(\gamma)e_j=\sum_{k=1}^{d_\pi}\alpha_{k,j}e_k$.
Moreover, the equation $\pi^*(\gamma)=\pi(\gamma)$ implies that $\overline{\alpha_{i,j}}=\alpha_{j,i}$. Now Equation \eqref{eq:1-last-prop}, together with the linear expansion of $\pi(\gamma)e_j$ given above, 
implies that
\begin{equation*}
T_w(\pi_{i,j})(y)=\sum_{k=1}^{d_\pi}\overline{\alpha_{k,j}}\langle \pi(y)e_i,e_k\rangle
=(\sum_{k=1}^{d_\pi}\overline{\alpha_{k,j}}\pi_{i,k})(y)
=(\sum_{k=1}^{d_\pi}\alpha_{j,k}\pi_{i,k})(y).
\end{equation*}
So for every $\pi\in \widehat{\bbG}$ and $1\leq i\leq d_\pi$, the set  ${\mathcal S}_{\pi, i}:={\rm span}\{\pi_{i,j}: 1\leq j\leq d_\pi\}$ is an invariant subspace for $T_w$. On the other hand, by Peter-Weyl Theorem, we have the Hilbert space decomposition $L^2(\bbG)\simeq \ell^2\text{-}\oplus_{\pi\in\widehat{\bbG}}\oplus_{i=1}^{d_\pi} {\mathcal S}_{\pi, i}$. Thus $T_w$ is block diagonalized when this Hilbert space decomposition is in place. 

We now proceed to diagonalize each block. 
For every $\pi\in\widehat{\bbG}$, we know by Lemma~\ref{lem:pi(f)sa} that $\pi(\gamma)$ is a self-adjoint matrix. So the same is true for the entry-wise complex conjugate matrix $\overline{\pi(\gamma)}$, and it can be diagonalized using its eigenbasis. 
Let $\lambda$ be an eigenvalue of $\overline{\pi(\gamma)}$, and 
suppose the nonzero vector $Z=[z_j]$ is a $\lambda$-eigenvector, i.e.~$\overline{\pi(\gamma)}Z=\lambda Z$. 
Then,  for $1\leq i\leq d_\pi$ we have,
\begin{eqnarray*}
T_w(\sum_{j=1}^{d_\pi}z_j\pi_{i,j})&=&\sum_{j=1}^{d_\pi}z_j\sum_{k=1}^{d_\pi}\alpha_{j,k}\pi_{i,k}
=\sum_{j=1}^{d_\pi}\left(\sum_{s=1}^{d_\pi} \alpha_{s,j}z_s\right)\pi_{i,j}\\
&=&\sum_{j=1}^{d_\pi}\left(\sum_{s=1}^{d_\pi} \overline{\alpha_{j,s}}z_s\right)\pi_{i,j}=\lambda\sum_{j=1}^{d_\pi} z_j \pi_{i,j},
\end{eqnarray*}
which proves that $\sum_{j=1}^{d_\pi}z_j\pi_{i,j}$ is a $\lambda$-eigenvector of $T_w$.
\end{proof}
As seen in Remark \ref{remark:leftregular}, for any $\phi\in L^2(\bbG)$, $T_w(\phi)=L(\gamma)\phi$.
This can be used to give a more abstract proof  of the previous proposition. We have avoided such abstract proofs in this paper, as the details of the unitary equivalences and the precise change of basis are important for graph signal processing applications. The following example demonstrates how the proposition can be used in such a setting.
\begin{example}[Ranking graphon]\label{exp:S3}
Consider the group of permutations on 3 elements:
$${\mathbb S}_3=\left\{g_1={\rm id},\  g_2=(12),\  g_3=(23),\  g_4=(13), \ g_5=(123),\  g_6=(132)\right\}.$$
The irreducible representations of ${\mathbb S}_3$ can be listed as follows:
\begin{itemize}
    \item[(i)] the trivial representation $\iota:{\mathbb S}_3\to \bbC$, defined as $\iota(g)=1$ for all $g\in {\mathbb S}_3$;
    \item[(ii)] the alternating representation $\tau:{\mathbb S}_3\to \bbC$, assigning to a permutation $g$ the sign of the permutation;
    \item[(iii)] the standard representation $\pi:{\mathbb S}_3\to \U(\bbC^2)$, defined as 
   $$\pi({\rm id})=\begin{bmatrix} 1 & 0\\ 0 & 1\end{bmatrix}, \ 
    \pi((12))=\begin{bmatrix} -\frac{1}{2} & \frac{\sqrt{3}}{2}\\ \frac{\sqrt{3}}{2} & \frac{1}{2}\end{bmatrix},\ 
    \pi((23))=\begin{bmatrix} 1 & 0\\ 0 & -1\end{bmatrix},\ 
    \pi((13))=\begin{bmatrix} -\frac{1}{2} & -\frac{\sqrt{3}}{2}\\ -\frac{\sqrt{3}}{2} & \frac{1}{2}\end{bmatrix},$$
    $$\pi((123))=\begin{bmatrix} -\frac{1}{2} & -\frac{\sqrt{3}}{2}\\ \frac{\sqrt{3}}{2} & -\frac{1}{2}\end{bmatrix},\ 
    \pi((132))=\begin{bmatrix} -\frac{1}{2} & \frac{\sqrt{3}}{2}\\ -\frac{\sqrt{3}}{2} & -\frac{1}{2}\end{bmatrix}.$$
\end{itemize}
As usual, we represent a complex-valued function on ${\mathbb S}_3$ by a vector in $\bbC^6$. Clearly, the (unique) coefficient function of every 1-dimensional representation is simply the representation itself. Equipping $\bbC^2$ with the standard basis 
$\left\{\begin{bmatrix} 1\\ 0\end{bmatrix}, \begin{bmatrix}  0 \\ 1\end{bmatrix}\right\}$, the coefficient functions associated with $\pi$ are given as follows
$$
    \pi_{1,1}=\begin{bmatrix} 1\\ -\frac{1}{2}\\ 1\\ -\frac{1}{2}\\ -\frac{1}{2}\\ -\frac{1}{2}\end{bmatrix}, \ 
    \pi_{2,1}=\begin{bmatrix} 0\\ \frac{\sqrt{3}}{2}\\ 0\\ -\frac{\sqrt{3}}{2}\\ -\frac{\sqrt{3}}{2}\\ \frac{\sqrt{3}}{2}\end{bmatrix}, \ 
    \pi_{1,2}=\begin{bmatrix} 0\\ \frac{\sqrt{3}}{2}\\ 0\\ -\frac{\sqrt{3}}{2} \\ \frac{\sqrt{3}}{2}\\ -\frac{\sqrt{3}}{2}\end{bmatrix}, \ 
    \pi_{2,2}=\begin{bmatrix} 1\\ \frac{1}{2}\\ -1\\ \frac{1}{2}\\ -\frac{1}{2}\\ -\frac{1}{2}\end{bmatrix}.
$$

\blue{The elements of ${\mathbb S}_3$ correspond to the different ways in which three distinct objects can be ranked. }
Consider the Cayley graphon $w:{\mathbb S}_3\times {\mathbb S}_3\to [0,1]$ defined by the Cayley function $\gamma:{\mathbb S}_3\to {\mathbb R}$, \blue{$\gamma=r\delta_{{\rm id}}+p\delta_{(12)}+q\delta_{(23)}$, where $0<p<q<r\leq 1$ and $\delta_g$ denotes the Dirac delta function. This is, in fact the same graphon as used in Example \ref{ex:S3_matrix}.

This graphon can be used to represent sets of individual entities that are distinguishable by the way they rank or prioritize three different items. For example, the entities could be political bloggers linked by `follow' relationships, and the labels from ${\mathbb S}_3$ represent the priority orderings the bloggers assign to a list of three election topics. A model for a graph corresponding to this situation would have the population divided into 6 groups, each labeled with an element of ${\mathbb S}_3$. 
Members of the same group have an identical ranking, which is the same as the group label. 
The link probability between entities is determined only by the set they belong to. The Cayley function $ \gamma=r\delta_{{\rm id}}+p\delta_{(12)}+q\delta_{(23)}$  where $0<p<q<r\leq 1$ then represents the following linking behaviour: entities in the same set are the most likely to link. Entities whose ranked list only differs in their choice for numbers 2 and 3 have second highest link probability. Entities whose ranked list transposes numbers 1 and 2 have a lower link probability. There are no links between groups whose list differ by more than an adjacent transposition.
The sampled graphs in Example \ref{ex:S3_matrix} conform to this model.}

Clearly, we have
$$\iota(\gamma)=\frac{r+p+q}{6},\  \tau(\gamma)=\frac{r-p-q}{6},  \ \pi(\gamma)=\frac{1}{6}\begin{bmatrix} r-\frac{p}{2}+q & \frac{\sqrt{3}p}{2}\\ \frac{\sqrt{3}p}{2} & r+\frac{p}{2}-q\end{bmatrix}.$$
(Here, we have normalized the counting measure on ${\mathbb S}_3$ to obtain a probability space.)
The eigenvalues of $\pi(\gamma)$ are $\frac{ 1}{6}(r\pm\sqrt{p^2+q^2-pq})$. From an easy calculation, we see 
$$\begin{bmatrix} -\frac{p-2q-2\sqrt{p^2-pq+q^2}}{\sqrt{3}p}\\1\end{bmatrix} \mbox{ and } \begin{bmatrix} -\frac{p-2q+2\sqrt{p^2-pq+q^2}}{\sqrt{3}p}\\1\end{bmatrix}$$ are eigenvectors of $\pi(\gamma)$ associated  with the positive and negative eigenvalues respectively.

Appealing to Theorem \ref{thm:eigenvector}, we conclude that the eigenvalues of $T_w$ are \blue{
$$\frac{r+p+q}{6}\ (\mbox{mult.~1}),\  \frac{r-p-q}{6}\ (\mbox{mult.~1}),\ \frac{1}{6}(r\pm\sqrt{p^2+q^2-pq})\ (\mbox{mult.~2 each}). 
$$}
Next, using Proposition~\ref{prop:basis}, we have the following set of eigenvectors for $T_w$, listed to correspond to the above set of eigenvalues. Note that in this case, we have $\pi(\gamma)=\overline{\pi(\gamma)}$, so the condition of Proposition~\ref{prop:basis} is satisfied.
Let $s=-\frac{p-2q-2\sqrt{p^2-pq+q^2}}{\sqrt{3}p}$ and $r=-\frac{p-2q+2\sqrt{p^2-pq+q^2}}{\sqrt{3}p}$.
$$
\iota=\begin{bmatrix} 1\\ 1\\ 1\\ 1\\ 1\\ 1 \end{bmatrix},\
\tau=\begin{bmatrix} 1\\ -1\\ -1\\ -1\\ 1\\ 1 \end{bmatrix}, \
s\pi_{1,1}+\pi_{1,2}= \begin{bmatrix} s\\ -\frac{s}{2}+\frac{\sqrt{3}}{2}\\ s\\ -\frac{s}{2}-\frac{\sqrt{3}}{2}\\ -\frac{s}{2}+\frac{\sqrt{3}}{2}\\ -\frac{s}{2}-\frac{\sqrt{3}}{2}\end{bmatrix}, \ 
s\pi_{2,1}+\pi_{2,2}= \begin{bmatrix} 1\\ \frac{\sqrt{3}s}{2}+\frac{1}{2}\\ -1\\ -\frac{\sqrt{3}s}{2}+\frac{1}{2}\\ -\frac{\sqrt{3}s}{2}-\frac{1}{2}\\ \frac{\sqrt{3}s}{2}-\frac{1}{2}\end{bmatrix},
$$
$$
r\pi_{1,1}+\pi_{1,2}= \begin{bmatrix} r\\ -\frac{r}{2}+\frac{\sqrt{3}}{2}\\ r\\ -\frac{r}{2}-\frac{\sqrt{3}}{2}\\ -\frac{r}{2}+\frac{\sqrt{3}}{2}\\ -\frac{r}{2}-\frac{\sqrt{3}}{2}\end{bmatrix}, \ 
r\pi_{2,1}+\pi_{2,2}= \begin{bmatrix} 1\\ \frac{\sqrt{3}r}{2}+\frac{1}{2}\\ -1\\ -\frac{\sqrt{3}r}{2}+\frac{1}{2}\\ -\frac{\sqrt{3}r}{2}-\frac{1}{2}\\ \frac{\sqrt{3}r}{2}-\frac{1}{2}\end{bmatrix}.
$$
\blue{In Example \ref{ex:S3_matrix}, $r=0.6$, $q=0.3$ and $p=0.1$. The second eigenvalue $\mu_2=\frac16(0.6+0.1\sqrt{7})$ has multiplicity 2, with eigenvectors $s\pi_{1,1}+\pi_{1,2}$ and $s\pi_{2,1}+\pi_{2,2}$ as given above. If this theoretical basis is used to compute the graph Fourier transform, then the Fourier coefficients corresponding to $\mu_2$ are as indicated by the red diamond in Figure \ref{fg:projections}. Using the theoretical basis gives a stable graph Fourier transform for samples of a Cayley graphon. }
\end{example}

Finally, we consider the special case for Cayley graphons where the Cayley function is constant on conjugacy classes. We refer to such graphons as \emph{quasi-Abelian} Cayley graphons; this terminology is an extension of a similar concept for Cayley graphs (\cite{Rockmore}). In this case, 
Proposition \ref{prop:basis} takes a greatly simplified form. Namely, the eigenbasis for $T_w$ derived from the irreducible representations of the group consists simply of all coefficient functions $\pi_{i,j}$. This result is a generalization of an analogue theorem for Cayley graphs; see \cite[Theorem 1.1]{Rockmore} or \cite[Theorem III.1]{sampta} for a proof. We state the result in the following corollary. 
\begin{corollary}\label{cor:constant-on-conjugacy}
Consider a compact group $\bbG$ together with a Cayley function $\gamma:\bbG\to [0,1]$ that is a class function, i.e.~ $\gamma$ is constant on conjugacy classes of $\bbG$ (or equivalently $\gamma(xy)=\gamma(yx)$ for all $x,y\in\bbG$). Let $w$ be the Cayley graphon associated with $\bbG$ and $\gamma$.
Then, for every $\pi\in \widehat{\bbG}$ and $1\leq i,j\leq d_\pi$,
\[ T_w(\pi_{i,j}) = \lambda_\pi \pi_{i,j},\]
where $\lambda_\pi = \frac{1}{d_\pi}{\rm Tr}(\pi(\gamma))$.
\end{corollary}
\begin{proof}
It is known that the set of  characters $\left\{\chi_\pi:=\sum_{i=1}^{d_{\pi}}\pi_{i,i}:\ \pi\in \widehat{\bbG}\right\}$ of a group $\bbG$ forms an orthonormal basis for the subspace of class functions in $L^2(\bbG)$ (see e.g.~\cite[Proposition 5.23]{1995:Folland:HarmonicAnalysis}).
Since $\gamma$ is a class function, we have
\begin{equation}\label{eq:expansion1}
    \gamma=\sum_{\pi\in\widehat{\bbG}}\langle \gamma, \chi_\pi\rangle_{{L^2(\bbG)}}\chi_\pi
    =\sum_{\pi\in\widehat{\bbG}}\sum_{i=1}^{d_\pi}\langle \gamma, \chi_\pi\rangle_{{L^2(\bbG)}}\pi_{i,i}.
\end{equation}
Let $\pi\in\widehat{\bbG}$ be arbitrary. Using Schur's orthogonality relations, Equation \eqref{eq:expansion1} implies that
$\langle \gamma, \pi_{i,j}\rangle_{{L^2(\bbG)}}=0$ if $i\neq j$, and 
$\langle \gamma, \pi_{i,i}\rangle_{{L^2(\bbG)}}=\frac{1}{d_\pi}\langle \gamma, \chi_\pi\rangle_{{L^2(\bbG)}}$. 
This allows us to compute the entries of the matrix $\pi(\gamma)$. Namely, since $\gamma$ is real-valued, we have
\begin{eqnarray*}
\langle \pi(\gamma)e_i,e_j\rangle = \int_{\bbG} \overline{\gamma(x)}\langle \pi(x)e_i,e_j\rangle \, dx= \overline{\langle\gamma, \pi_{i,j}\rangle}_{L^2(\bbG)}
=\left\{\begin{array}{cc}
  0   & i\neq j \\
  \frac{1}{d_\pi}\langle\chi_\pi, \gamma\rangle_{{L^2(\bbG)}}   & i=j
\end{array}\right..
\end{eqnarray*}
In other words, $\pi(\gamma)e_i= \frac{1}{d_\pi}\langle\chi_\pi, \gamma\rangle_{{L^2(\bbG)}} e_i$ for every $1\leq i\leq d_{\pi}$.
Conjugating both sides of the previous equation, we conclude that the standard basis $\{e_i\}_{i=1}^{d_{\pi}}$ is an orthonormal eigenbasis of $\overline{\pi(\gamma)}$ associated with the eigenvalue 
$\frac{1}{d_\pi}\langle \gamma, \chi_\pi\rangle_{{L^2(\bbG)}}$. So by Proposition \ref{prop:basis}, $\cup_{\pi\in\widehat{\bbG}}\{\pi_{i,j}: 1\leq i,j\leq d_\pi\}$ forms an orthogonal eigenbasis for $T_w$ associated with (repeated) eigenvalues $\frac{1}{d_\pi}\langle \gamma, \chi_\pi\rangle_{{L^2(\bbG)}}$. Finally, observe that
\begin{equation*}
\langle \gamma, \chi_\pi\rangle_{{L^2(\bbG)}}=\sum_{i=1}^{d_\pi}\int_{\bbG}\gamma(x)\overline{\pi_{i,i}(x)}\,  dx =\sum_{i=1}^{d_\pi}\int_{\bbG}\gamma(x)\langle\pi(x)e_{i},e_i\rangle \, dx
=\sum_{i=1}^{d_\pi}\langle\pi(\gamma)e_i,e_i\rangle={\rm Tr}(\pi(\gamma)),
\end{equation*}
which finishes the proof. 
\end{proof}

\begin{example}[SO(3)]
Consider the (non-Abelian) group SO(3) of all rotations of the unit ball around an axis through the origin. Thus, each element of SO(3) can be characterized by a unit vector indicating the axis, and a rotation angle. It is well-known that two elements of SO(3) are conjugate if and only if they have the same rotation angle. Thus, if we let the Cayley function $\gamma$ be any function that depends only on the rotation angle, then $\gamma$ satisfies the conditions of Corollary \ref{cor:constant-on-conjugacy}. The corollary now tells us that the coefficient functions $\pi_{i,j}$ provide an eigenbasis for the graphon, which can be used to define a Fourier transform for graphs sampled from the graphon.

A natural Cayley
graphon results if we let $\gamma$ be a sharply declining function of the rotation angle. In that case, two rotations $\sigma$ and $\tau$ in SO(3) have high link probability if $\sigma\tau^{-1}$ has a very small angle. This can be interpreted as $\sigma$ and $\tau$ having a similar effect on the unit ball.
\end{example}

\section{Acknowledgements}
The first author was supported by NSF grant DMS-1902301, while this work was being completed. The second author was supported by NSERC. 
The first two authors initiated  this project while on a Research in Teams visit at the Banff International Research Center. They are
grateful to BIRS for the financial support and hospitality.



\begin{thebibliography}{10}

\bibitem{Abbe}
Emmanuel Abbe.
\newblock Community detection and stochastic block models: recent developments.
\newblock {\em J. Mach. Learn. Res.}, 18:Paper No. 177, 86, 2017.

\bibitem{Bollobas-book}
B\'{e}la Bollob\'{a}s.
\newblock {\em Linear analysis}.
\newblock Cambridge Mathematical Textbooks. Cambridge University Press,
  Cambridge, 1990.
\newblock An introductory course.

\bibitem{Borg-Chayes-Lovasz-2010}
Christian Borgs, Jennifer Chayes, and L\'{a}szl\'{o} Lov\'{a}sz.
\newblock Moments of two-variable functions and the uniqueness of graph limits.
\newblock {\em Geom. Funct. Anal.}, 19(6):1597--1619, 2010.

\bibitem{borgsI2008}
Christian Borgs, Jennifer~T. Chayes, L{\'a}szl{\'o} Lov{\'a}sz, Vera S{\'o}s,
  and Katalin Vesztergombi.
\newblock Convergent sequences of dense graphs {I}. {S}ubgraph frequencies,
  metric properties and testing.
\newblock {\em Adv. Math.}, 219(6):1801--1851, 2008.

\bibitem{BCLSV2011}
Christian Borgs, Jennifer~T. Chayes, L{\'a}szl{\'o} Lov{\'a}sz, Vera S{\'o}s,
  and Katalin Vesztergombi.
\newblock Limits of randomly grown graph sequences.
\newblock {\em European J. Combin.}, 32(7):985--999, 2011.

\bibitem{permutahedron}
Yilin Chen, Jennifer DeJong, Tom Halverson, and David~I. Shuman.
\newblock Signal {P}rocessing on the {P}ermutahedron: {T}ight {S}pectral
  {F}rames for {R}anked {D}ata {A}nalysis.
\newblock {\em J. Fourier Anal. Appl.}, 27(4):Paper No. 70, 2021.

\bibitem{1995:Folland:HarmonicAnalysis}
Gerald~B. Folland.
\newblock {\em A course in abstract harmonic analysis}.
\newblock Studies in Advanced Mathematics. CRC Press, Boca Raton, FL, 1995.

\bibitem{cut-norm}
Alan Frieze and Ravi Kannan.
\newblock Quick approximation to matrices and applications.
\newblock {\em Combinatorica}, 19(2):175--220, 1999.

\bibitem{GGH-frames}
Mahya Ghandehari, Dominique Guillot, and Kris Hollingsworth.
\newblock Gabor-type frames for signal processing on graphs.
\newblock {\em J. Fourier Anal. Appl.}, 27(2):Paper No. 25, 23, 2021.

\bibitem{sampta}
Mahya Ghandehari, Dominique Guillot, and Kristopher Hollingsworth.
\newblock A non-commutative viewpoint on graph signal processing.
\newblock {\em IEEE Xplore: 2019 {I}nternational {C}onference on {S}ampling
  {T}heory and {A}pplications}, 2019.

\bibitem{HuangGG09}
Jonathan Huang, Carlos Guestrin, and Leonidas~J. Guibas.
\newblock Fourier theoretic probabilistic inference over permutations.
\newblock {\em J. Mach. Learn. Res.}, 10:997--1070, 2009.

\bibitem{LiWX19}
Xue Li, Xinlei Wang, and Guanghua Xiao.
\newblock A comparative study of rank aggregation methods for partial and top
  ranked lists in genomic applications.
\newblock {\em Briefings Bioinform.}, 20(1):178--189, 2019.

\bibitem{lovasz-book}
L{\'a}szl{\'o} Lov{\'a}sz.
\newblock {\em Large networks and graph limits}, volume~60 of {\em American
  Mathematical Society Colloquium Publications}.
\newblock American Mathematical Society, Providence, RI, 2012.

\bibitem{lovaszszegedy2006}
L\'{a}szl\'{o} Lov\'{a}sz and Bal\'{a}zs Szegedy.
\newblock Limits of dense graph sequences.
\newblock {\em J. Combin. Theory Ser. B}, 96(6):933--957, 2006.

\bibitem{cayley-graphon}
L\'{a}szl\'{o} Lov\'{a}sz and Bal\'{a}zs Szegedy.
\newblock The automorphism group of a graphon.
\newblock {\em J. Algebra}, 421:136--166, 2015.

\bibitem{neural2}
Sohir Maskey, Ron Levie, and Gitta Kutyniok.
\newblock Transferability of graph neural networks: an extended graphon
  approach.
\newblock {arXiv}:2109.10096, 2021.

\bibitem{MorencyLeus17}
Matthew~W. Morency and Geert Leus.
\newblock Signal processing on kernel-based random graphs.
\newblock In {\em 2017 25th European Signal Processing Conference (EUSIPCO)},
  pages 365--369, 2017.

\bibitem{MorencyLeus21}
Matthew~W. Morency and Geert Leus.
\newblock Graphon filters: Graph signal processing in the limit.
\newblock {\em IEEE Transactions on Signal Processing}, 69:1740--1754, 2021.

\bibitem{2018:Ortega:GSPOverview}
Antonio Ortega, Pascal Frossard, Jelena Kovačević, Jos\'{e} M.~F. Moura, and
  Pierre Vandergheynst.
\newblock Graph signal processing: Overview, challenges, and applications.
\newblock {\em Proceedings of the IEEE}, 106(5):808--828, 2018.

\bibitem{Rockmore}
Dan Rockmore, Peter Kostelec, Wim Hordijk, and Peter~F. Stadler.
\newblock Fast {F}ourier transform for fitness landscapes.
\newblock {\em Appl. Comput. Harmon. Anal.}, 12(1):57--76, 2002.

\bibitem{Rockmore1}
Daniel~N. Rockmore.
\newblock Some applications of generalized {FFT}s.
\newblock In {\em Groups and computation, {II} ({N}ew {B}runswick, {NJ},
  1995)}, volume~28 of {\em DIMACS Ser. Discrete Math. Theoret. Comput. Sci.},
  pages 329--369. Amer. Math. Soc., Providence, RI, 1997.

\bibitem{ruiz2}
Luana Ruiz, Luiz F.~O. Chamon, and Alejandro Ribeiro.
\newblock The graphon {F}ourier transform.
\newblock In {\em ICASSP 2020 - 2020 IEEE International Conference on
  Acoustics, Speech and Signal Processing (ICASSP)}, pages 5660--5664, 2020.

\bibitem{RuizChamonRibeiro21}
Luana Ruiz, Luiz F.~O. Chamon, and Alejandro Ribeiro.
\newblock Graphon signal processing, 2021.
\newblock To appear in IEEE Trans.~Signal Processing.

\bibitem{neural1}
Luana Ruiz, Luiz F.~O. Chamon, and Alejandro Ribeiro.
\newblock Transferability properties of graph neural networks.
\newblock {arXiv}:2112.04629, 2021.

\bibitem{2013:Sandryhaila:DS}
Aliaksei Sandryhaila and Jos\'{e} M.~F. Moura.
\newblock Discrete signal processing on graphs.
\newblock {\em IEEE Trans. Signal Process.}, 61(7):1644--1656, 2013.

\bibitem{2014:Sandryhaila:BD}
Aliaksei Sandryhaila and Jos\'{e} M.~F. Moura.
\newblock Big data analysis with signal processing on graphs: Representation
  and processing of massive data sets with irregular structure.
\newblock {\em IEEE Signal Processing Magazine}, 31(5):80--90, Sept 2014.

\bibitem{SandryhailaMoura13}
Aliaksei Sandryhaila and Jos\'{e} M.~F. Moura.
\newblock Discrete signal processing on graphs: frequency analysis.
\newblock {\em IEEE Trans. Signal Process.}, 62(12):3042--3054, 2014.

\bibitem{SNFOV}
David~I. Shuman, Sunil~K. Narang, Antonio~Ortega Pascal~Frossard, and Pierre
  Vandergheynst.
\newblock The emerging field of signal processing on graphs: Extending
  high-dimensional data analysis to networks and other irregular domains.
\newblock {\em IEEE Signal Processing Magazine.}, 30(5):83--98, 2013.

\bibitem{Szegedy-spectra}
Bal\'{a}zs Szegedy.
\newblock Limits of kernel operators and the spectral regularity lemma.
\newblock {\em European J. Combin.}, 32(7):1156--1167, 2011.

\bibitem{Terras:1999:FourierAnalysisOnFiniteGroups}
Audrey Terras.
\newblock {\em Fourier Analysis on Finite Groups and Applications}.
\newblock Cambridge University Press, 1999.

\bibitem{UminskyBGGN19}
David Uminsky, Mario Banuelos, Lillian González-Albino, Rosa Garza, and
  Sylvia~Akueze Nwakanma.
\newblock Detecting higher order genomic variant interactions with spectral
  analysis.
\newblock In {\em EUSIPCO}, pages 1--5. IEEE, 2019.

\bibitem{WangSE14}
Jialei Wang, Nathan Srebro, and James Evans.
\newblock Active collaborative permutation learning.
\newblock In Sofus~A. Macskassy, Claudia Perlich, Jure Leskovec, Wei Wang, and
  Rayid Ghani, editors, {\em KDD}, pages 502--511. ACM, 2014.

\bibitem{WattsStrogatz98}
Duncan Watts and Steven Strogatz.
\newblock Collective dynamics of {"small-world"} networks.
\newblock {\em Nature}, 393:440--442, 1998.

\bibitem{model-fitting}
Hoi~Sim Wong, Tat-Jun Chin, Jin Yu, and David Suter.
\newblock Mode seeking over permutations for rapid geometric model fitting.
\newblock {\em Pattern Recogn.}, 46(1):257--271, 2013.

\end{thebibliography}
\end{document}